\documentclass[11pt]{article}

\usepackage[a4paper,margin=1in]{geometry}
\usepackage{amsmath}

\usepackage[utf8]{inputenc}
\usepackage{hyperref}
\usepackage{xcolor}
\usepackage{enumitem}
\usepackage{amssymb}
\usepackage{geometry}
\usepackage{enumitem}
\usepackage{multirow}
\usepackage{amsmath}
\usepackage{amsthm}
\usepackage{cite}
\usepackage{caption}
\usepackage{subcaption}
\usepackage{graphicx}
\newtheorem{prop}{Proposition}
\newtheorem{theorem}{Theorem}
\newtheorem{lemma}{Lemma}

\newtheorem{example}{Example}

\usepackage{tikz}

\definecolor{lightgreen}{rgb}{0.439,0.678,0.278}
\definecolor{lightblue}{rgb}{0.267,0.447,0.769}

\def\X{{\mathcal X}}
\def\C{{\mathcal C}}
\def\E{{\mathcal E}}
\def\G{{\mathcal G}}
\def\ie{{i.e.,} }
\def\prob{{\mathbb P}}
\def\MP{\text{MP}}
\def\LP{\text{LP}}
\def\st{{\rm s.t.}}
\def\NP{{\mathcal NP}}

\def\proj{{\rm proj}}
\def\conv{{\rm conv}}
\def\D{{\mathcal D}}
\def\P{{\mathcal P}}
\def\H{{\mathcal H}}

\allowdisplaybreaks

\title{Inference in higher-order undirected graphical models and binary polynomial optimization
\thanks{Both authors were partially funded by AFOSR grant FA9550-23-1-0123.}}

\author{Aida Khajavirad, Yakun Wang}
\date{}

\author{Aida Khajavirad
\thanks{Department of Industrial and Systems Engineering,
             Lehigh University.
             E-mail: {\tt aida@lehigh.edu}.
             }
\and
Yakun Wang
\thanks{Department of Industrial and Systems Engineering,
             Lehigh University.
             E-mail: {\tt yaw220@lehigh.edu}.
             }
}

\begin{document}

\maketitle

\begin{abstract}
We consider the problem of inference in higher-order undirected graphical models with binary labels. We formulate this problem as a binary polynomial optimization problem and propose several linear programming relaxations for it. We compare the strength of the proposed linear programming relaxations theoretically. Finally, we demonstrate the effectiveness of these relaxations by performing a computational study for two important applications, namely, image restoration and decoding error-correcting codes.
\end{abstract}


\section{Introduction}\label{introduction}
Graphical models are powerful probabilistic modeling tools
for capturing complex relationships among large collections of random variables and have found ample applications in  computer vision, natural language processing, signal processing, bioinformatics, and statistics~\cite{WaiJor08}.
In this framework, dependencies among random variables are represented by a graph. If this graph is a directed acyclic graph, the graphical model is often referred to as a \emph{Bayesian Network}, while if the graph is undirected, it is often referred to as an \emph{Undirected Graphical Model} (UGM) or a~\emph{Markov Random Field}.
In this paper, we focus on UGMs.

\paragraph{Undirected graphical models.} Let $G=(V,E)$ be an undirected graph, where $V, E$ denote node set and edge set of $G$, respectively. In order to define a graphical model, we associate with each node
 $v  \in V$ a random variable $X_v$ taking values in some state space $\X_v$. The edge set $E$ represents dependencies between random variables; that is, for any three distinct nodes $u,v,w \in V$, $X_u$ is independent of $X_v$ given $X_w$, if every path from $u$ to $v$ in $G$ passes through $w$. The notation $\prob(X_v = x_v)$ corresponds to the probability of the event that the random variable $X_v$ takes the value $x_v \in \X_v$.
 Denote by $\C$ the set of maximal cliques in $G$;  \ie the set of cliques that are not properly contained in any other clique of $G$. Recall that a clique $C$ is a subset of $V$ such that $\{u,v\} \in E$ for all $u \neq v \in C$. For each $C \in \C$, let us define a nonnegative \emph{potential function} $\phi_C(x_C)$, where $x_C$ is the vector consisting of $x_v$, $v \in C$.  In this paper, we assume $\X_v = \{0,1\}$ for all $v \in V$, henceforth referred to as a binary UGM.
 It can be shown that the joint probability mass function for a binary UGM is given by:
 \begin{equation}\label{pmf}
 p(x_v, v \in V) = \prob(X_v=x_v, v\in V)= \frac{1}{Z} \prod_{C \in \C}{\phi_C(x_C)},
 \end{equation}
where $Z$ is a normalization constant given by $Z:= \sum_{x}{\prod_{C\in \C}\phi_C(x_C)}$. The \emph{order} of a UGM is defined as the size of a largest clique $C \in \C$ minus one.
Due to their simplicity, first-order UGMs, also known as pair-wise models, are the most popular UGMs. However, to model more complex interactions among random variables, it is essential to study higher-order UGMs.

\paragraph{Inference in binary UGMs.}
Given some noisy observation $Y_v$, $v \in V$, we would like to recover the ground truth $X_v$, $v \in V$, whose probability mass function is described by a binary UGM defined by~\eqref{pmf}. By definition,
the maximum a posteriori (MAP) estimator maximizes the probability of recovering the ground truth.
In the following, we denote by $\prob[X|Y]$ the probability that $X_v$, $v \in V$ is the ground truth, given that $Y_v$, $v \in V$ is observed. Hence, we are interested in solving the following optimization problem:
\begin{equation}\label{map1}
{\rm max}_{x \in \{0,1\}^V} \; \prob[X|Y].
\end{equation}
Notice that $\prob[Y] > 0$, and does not depend on $x$. Hence, by Bayes' theorem and monotonicity of the log function, we deduce that
$${\rm argmax}(\prob[X|Y])= {\rm argmax}(\prob[Y|X] \prob[X])= {\rm argmax}(\log(\prob[Y|X] \prob[X])).$$
Using~\eqref{pmf}, it follows that to solve Problem~\ref{map1} we can equivalently solve:
\begin{equation}\label{map2}
    {\rm max}_{x \in \{0,1\}^V} \log(\prob[Y|X])+\sum_{C \in \C}{\log(\phi_C(x_C))}.
\end{equation}
As we mentioned before, in this paper we consider binary UGMs. Suppose that $\phi_C(x_C) > 0$ for all $C \in \C$.
It is well-known that any real-valued function in binary variables can be written as a binary polynomial function in the same variables. Given a clique $C$, denote by $P(C)$ the power set of $C$. Then Problem~\ref{map2} can be equivalently written as:
\begin{equation}\label{map3}
    {\rm max}_{x \in \{0,1\}^V} \log(\prob[Y|X])+\sum_{C \in \C}{\sum_{e \in P(C)} c_e \prod_{v \in e} x_v}.
\end{equation}
We should remark that if $\phi_C(x_C) = 0$ for some $x_C \in \X_C' \subset \{0,1\}^C$, then one can proceed by adding the constraint $x_C \not\in \X'_C$ to Problem~\ref{map2}. We use this technique in Section~\ref{sec:decoding} to formulate the decoding problem.
Now let us consider the first term in the objective function of Problem~\ref{map3}. To obtain an explicit description for $\prob[Y|X]$, we have to make assumptions on the noise. In the following we introduce a simple noise model that we will also use in our numerical experiments.
Given $p \in (0, \frac{1}{2}]$, the noisy observation $Y$ is constructed as follows: for each $v \in V$, $y_v$ is corrupted with probability $p$, \ie $y_v = 1-x_v$, and $y_v$ is not corrupted with probability $1-p$, \ie $y_v = x_v$. We refer to this noise model as the \emph{bit-flipping} noise. We then have
\begin{align*}
\prob\left[y_v \mid x_v\right]=
\begin{cases}
p^{x_v} (1-p)^{1-x_v} & \text{if $y_v=0$} \\
p^{1-x_v} (1-p)^{x_v} & \text{if $y_v=1$}.
\end{cases}
\end{align*}
Since the probability of corruption of the entries of $Y$ are independent, we have
$$
\prob[Y \mid X]
=\prod_{v \in V} \prob\left[y_v \mid x_v\right]
 =\prod_{v \in V: y_v = 0} {p^{x_v} (1-p)^{1-x_v}} \prod_{v \in V: y_v = 1} {p^{1-x_v} (1-p)^{x_v}},
$$
which in turn implies that
\begin{align*}
\log(\prob[Y \mid X])
& = \log\Big(\frac{1-p}{p}\Big)\Big(\sum_{v\in V: y_v=1}{x_v}-\sum_{v\in V: y_v=0}{x_v}\Big)+\Big|\{v\in V: y_v = 1\}\Big|\log p\\
&+ \Big|\{v \in V: y_v=0\}\Big|\log(1-p) .
\end{align*}
We then deduce that under the bit-filliping noise, to solve Problem~\ref{map3}, it suffices to solve the following \emph{unconstrained binary polynomial optimization problem}:
\begin{equation}\label{map4}
    {\rm max}_{x \in \{0,1\}^V} \log\Big(\frac{1-p}{p}\Big)\Big(\sum_{v\in V: y_v=1}{x_v}-\sum_{v\in V: y_v=0}{x_v}\Big)+\sum_{C \in \C}{\sum_{e \in P(C)} c_e \prod_{v \in e} x_v}.
\end{equation}
Note that $\log\big(\frac{1-p}{p}\big) > 0$, since by assumption $p \in (0,\frac{1}{2}]$. An optimal solution of Problem~\ref{map4} is a MAP estimator under the bit-flipping noise and it requires parameter $p$ as an input. We would like to employ a formulation that does not have the knowledge of how the noisy observation was generated. That is, we propose to solve the following optimization problem:
\begin{equation}\label{map5}
    {\rm max}_{x \in \{0,1\}^V} \alpha\Big(\sum_{v\in V: y_v=1}{x_v}-\sum_{v\in V: y_v=0}{x_v}\Big)+\sum_{C \in \C}{\sum_{e \in P(C)} c_e \prod_{v \in e} x_v},
\end{equation}
where $\alpha$ is a positive parameter that along with the remaining parameters $c_e$, $e \in P(C)$, $C \in \C$ are \emph{learned} from the data.

\paragraph{Literature review.} The literature on inference in UGMs is mostly focused on first-order UGMs; \ie the case where $|C|=2$ for all $C \in \C$. For first-order binary UGMs, Problem~\ref{map5} simplifies to an unconstrained binary quadratic optimization problem, which is NP-hard in general. The most popular methods to tackle this inference problem are belief propagation~\cite{wainwright2005map,felzenszwalb2006efficient}, which is a message passing algorithm, and graph cut algorithms~\cite{kolmogorov2004energy,boykov2006graph,boykov2001fast,kolmogorov2007minimizing}. Moreover, constant-factor approximation algorithms are available for this problem class~\cite{GoeWil95,nest98}.
Utilizing higher-order UGMs is essential for capturing more complex interactions among random variables. Yet, their study has been fairly limited due to the complexity of solving Problem~\ref{map5} in its full generality. In fact, almost all existing studies
considering higher-order UGMs tackle the inference problem by first reducing it to a binary quadratic optimization problem through the introduction of auxiliary variables and subsequently employing graph cut algorithms to solve the quadratic optimization problem~\cite{fix2011graph,rother2009minimizing,ishikawa2009higher,ishikawa2010transformation}. In~\cite{feldman2005using}, the authors consider the inference problem for a higher-order binary UGM arising from the error-correcting decoding problem and propose a linear programming (LP) relaxation for this problem.
In~\cite{feldman2006lp}, the authors analyze the performance of the LP relaxation of~\cite{feldman2005using} theoretically, hence establishing the effectiveness of the LP relaxation for decoding low-density-parity-check codes.
In~\cite{crama2017class} the authors consider a third-order binary UGM for a simplified image restoration problem and propose an LP relaxation for this problem.

\paragraph{Our contributions.} In spite of its ample applications, the existing results for inference in higher-order binary UGMs are rather scarce. In this paper, by building upon recent theoretical and algorithmic developments for binary polynomial optimization~\cite{dPKha17MOR,dPKha18MPA,dPKha18SIOPT,dPKha21MOR,dPKhaSah20MPC,dPKha23mMPA}, we present strong LP relaxations for Problem~\ref{map5} in its full generality. We prove that the proposed LPs are stronger than the existing LPs for this class of problems and can be solved efficiently using off-the-shelf LP solvers. We consider two important applications of inference in higher-order binary UGMs; namely image restoration, a popular application in computer vision, and decoding error-correcting codes, a central problem in information theory. Via an extensive computational study, we show that a simple LP relaxation that we refer to as the ``clique LP'' is often sharp for image restoration problems. The decoding problem on the other hand turns out to be a difficult problem and while the proposed clique LP outperforms the only existing LP relaxation for this problem~\cite{feldman2006lp}, the improvement is rather small.

\paragraph{Organization.} The remainder of this paper is structured as follows. In Section~\ref{sec:lprelaxations} we review existing LP relaxations for Problem~\ref{map5} and propose new LP relaxations for it. In Section~\ref{sec:images} we consider the image restoration problem  while in Section~\ref{sec:decoding} we consider the problem of decoding error-correcting codes.

\section{Linear programming relaxations}
\label{sec:lprelaxations}

With the objective of constructing LP relaxations for Problem~\ref{map5}, following a common practice in nonconvex optimization, we start by linearizing its objective function.
Define $\bar P(C) := P(C) \setminus (C \cup \{\emptyset\})$ for all $C \in \C$.
For notational simplicity, henceforth we denote variables $x_v$, $v \in V$ by $z_v$, $v \in V$. Define an auxiliary variable $z_e := \prod_{v\in e}{z_v}$
for all $e \in \bar P(C)$ and for all $C \in \C$. Then an equivalent reformulation of Problem~\ref{map5} in a lifted space of variables is given by:
\begin{align}\label{map6}\tag{B-UGM}
    {\rm max} \quad & \alpha \Big(\sum_{v\in V: y_v=1}{z_v}-\sum_{v\in V: y_v=0}{z_v}\Big)+\sum_{v\in V}{c_v z_v}+\sum_{C \in \C}{\sum_{e \in \bar P(C)} c_e z_e}\\
    \st \quad & z_e = \prod_{v \in e} {z_v}, \; \forall e \in \bar P(C), \; \forall C \in \C\nonumber\\
    & z_v \in \{0,1\}, \; \forall v \in V \nonumber.
\end{align}
Define the hypergraph $\G=(V,\E)$ with the node set $V$ and the edge set $\E:=\cup_{C \in \C}{\bar P(C)}$. The \emph{rank} of $\G$ is defined as the maximum cardinality of any edge in $\G$.
Following the convention first introduced in~\cite{dPKha17MOR}, we define the \emph{multilinear set} as:
\begin{equation}\label{multSet}
S(\G)=\Big\{z\in \{0,1\}^{V\cup\E}: z_e = \prod_{v \in e} {z_v}, \; \forall e \in \bar P(C), \; \forall C \in \C\Big\}.
\end{equation}
We refer to the convex hull of $S(\G)$ as the \emph{multilinear polytope} and denote it by $\MP(\G)$. Henceforth, we refer to the hypergraph $\G$ associated with Problem~\ref{map6} as an \emph{UGM hypergraph}.
To obtain an LP relaxation for Problem~\ref{map6}, it suffices to construct a polyhedral relaxation for the multilinear set $S(\G)$. Notice that the rank of a UGM hypergraph equals the size of the largest clique $C$ in the corresponding UGM; this number in turn is always quite small in practice. This key property enables us to obtain strong and yet cheaply computable relaxations for the multilinear polytope of a UGM hypergraph.
In the remainder of this section, we briefly review existing LP relaxations for Problem~\ref{map6}; subsequently, we propose new LP relaxations for it.

\subsection{The standard linearization}
\label{sec:stdLP}

The simplest and perhaps the oldest technique to convexify the multilinear set $S(\G)$ is to replace the feasible region defined by each product term $z_{e} = \prod_{v \in e} z_{v}$
over the set of binary points with its convex hull. We then obtain our first polyhedral relaxation of $S(\G)$:
\begin{align}\label{stdre}
\MP^{\LP}(\G) =\Big\{z: \; & z_v \leq 1, \; \forall v \in V,
z_e  \geq 0, \; z_e \geq \sum_{v\in e}{z_v}-|e|+1, \; \forall e \in \bar P(C), \forall C \in \C,\nonumber\\
 & z_e  \leq z_v, \forall v \in e, \; e \in \bar P(C), \forall C \in \C \Big\}.
\end{align}
The above relaxation has been used extensively in the literature and is often referred to as the~\emph{standard linearization}
of the multilinear set (see for example~\cite{gw74,Cra93}).
We then define our first LP relaxation which we refer to as the \emph{standard LP}:
\begin{align}\label{stdLP}\tag{stdLP}
    \max \quad &  \alpha \Big(\sum_{v\in V: y_v=1}{z_v}-\sum_{v\in V: y_v=0}{z_v}\Big)+\sum_{v\in V}{c_v z_v}+\sum_{C \in \C}{\sum_{e \in \bar P(C)} c_e z_e} \\
     \st \quad & z \in   \MP^{\LP}(\G). \nonumber
\end{align}
In~\cite{dPKha18SIOPT,BucCraRod18}, the authors prove that $\MP^{\LP}(\G)  = \MP(\G)$ if and only if $\G$ is a Berge-acyclic hypergraph; \ie the most restrictive type of acyclic hypergraphs~\cite{fag83}. However, a UGM hypergraph is not Berge-acyclic and indeed
our numerical experiments indicate that Problem~\ref{stdLP} leads to very weak upper bounds for Problem~\ref{map6}.

\subsection{The flower relaxation}
In~\cite{dPKha18SIOPT}, the authors introduce flower inequalities, a family of valid inequalities for the multilinear polytope, which strengthens the standard linearization. Flower inequalities were later generalized in~\cite{Kha22} and in the following we use this more general definition. Let $e_0, e_k$, $k \in T$, $T \neq \emptyset$ be edges of $\G$ such that
\begin{equation}\label{cond}
    \Big|(e_0 \cap e_k) \setminus \bigcup_{j \in T \setminus \{k\}}{(e_0\cap e_j)}\Big| \geq 2, \quad \forall k \in T.
\end{equation}
Then the~\emph{flower inequality} centered at $e_0$ with neighbors $e_k$, $k \in T$, is given by:
\begin{equation}\label{flowerIneq}
\sum_{v \in e_0\setminus \cup_{k\in T} {e_k}}{z_v}+\sum_{k \in T} {z_{e_k}} - z_{e_0} \leq |e_0\setminus \cup_{k\in T} {e_k}|+|T|-1.
\end{equation}
If $|T|=1$, flower inequalities simplify to two-link inequalities introduced in~\cite{crama2017class}.
We define the \emph{flower relaxation} $\MP^F(\G)$ as the polytope obtained by adding all flower inequalities centered at each edge of $\G$ to $\MP^{\LP}(\G)$.
In~\cite{dPKhaSah20MPC}, the authors prove that while the separation problem over the flower relaxation is $\NP$-hard for general hypergraphs, it can be solved in polynomial time for hypergraphs with fixed rank. As we discussed before, in case of a UGM hypergraph, it is reasonable to assume that the rank is fixed. In fact, as we show next, for a UGM hypergraph, it suffices to include only a small number of flower inequalities in the flower relaxation.
%
\begin{lemma}\label{lem:redfl}
    Let $\G=(V,\E)$ with $\E=\cup_{C \in \C}{\bar P(C)}$ be a UGM hypergraph and consider the flower relaxation $\MP^F(\G)$ of $S(\G)$. Denote by
    $\MP^{F'}(\G)$ the polytope obtained by adding all flower inequalities~\eqref{flowerIneq} satisfying
    \begin{equation}\label{nonred}
        \bigcup_{k \in T}{e_k} \cup e_0 \subseteq C, \quad {\rm for \; some} \; C \in \C,
    \end{equation}
    to the standard linearization. Then $\MP^F(\G) = \MP^{F'}(\G)$.
\end{lemma}
\begin{proof}
By construction, $\MP^F(\G) \subseteq \MP^{F'}(\G)$. Hence, it suffices to show that  $\MP^{F'}(\G) \subseteq \MP^{F}(\G)$.
To this end, we show that any flower inequality not satisfying condition~\eqref{nonred} is implied by some flower inequalities that satisfy this condition. Let $e_0, e_k$, $k \in T$ be edges of $\G$ satisfying condition~\eqref{cond} but not satisfying condition~\eqref{nonred}. Then the following flower inequality is present in $\MP^F(\G)$:
\begin{equation}\label{flowerred}
\sum_{v \in e_0\setminus \cup_{k\in T} {e_k}}{z_v}+\sum_{k \in T} {z_{e_k}} - z_{e_0} \leq |e_0\setminus \cup_{k\in T} {e_k}|+|T|-1.
\end{equation}
Define $\gamma(e_0, e_k, k\in T):= \bigcup_{k \in T}{e_k} \cup e_0$.
First consider the case where $e_j \supset e_0$ for some $j \in T$; then by~\eqref{cond} we must have $|T|= 1$, implying that $\gamma(e_0, e_j) = e_j$ and hence condition~\eqref{nonred} holds. Henceforth, suppose that $e_k \not\supset e_0$ for all $k \in T$. Notice that if $e_k \subset e_0$ for all $k \in T$, then we have $\gamma(e_0, e_k, k\in T) = e_0$ and condition~\eqref{nonred} is trivially satisfied. Denote by $T'$ the nonempty set containing all $k \in T$ satisfying $e_k \setminus e_0 \neq \emptyset$. Define $\tilde e_k = e_k \cap e_0$ for all $k \in T'$. By definition of a UGM hypergraph and condition~\eqref{cond}, we have $\tilde e_k \in \E$ for all $k \in T'$. Hence the following flower inequalities are also present in $\MP^F(\G)$:
\begin{align}
 \sum_{v \in e_0\setminus \cup_{k\in T} {e_k}}&{z_v}+\sum_{k \in T \setminus T'} {z_{e_k}} +\sum_{k \in T'}{z_{\tilde e_k}}- z_{e_0} \leq |e_0\setminus \cup_{k\in T} {e_k}|+|T|-1\label{first}\\
& z_{e_k}-z_{\tilde e_k} \leq 0, \quad \forall k \in T',\label{second}
\end{align}
where we used the identity
$e_0\setminus \cup_{k\in T} {e_k} = e_0 \setminus ((\cup_{k\in T \setminus T'}{e_k}) \cup (\cup_{k\in T'}{\tilde e_k}))$.
First, it is simple to verify that condition~\eqref{nonred} is satisfied for inequalities~\eqref{first} and~\eqref{second}. Second, summing up inequalities~\eqref{first} and~\eqref{second}, we obtain inequality~\eqref{flowerred}, implying its redundancy, and this completes the proof.
\end{proof}
We then define our next LP relaxation, which we refer to as the \emph{flower LP}:
\begin{align}\label{flLP}\tag{flLP}
    \max \quad &  \alpha \Big(\sum_{v\in V: y_v=1}{z_v}-\sum_{v\in V: y_v=0}{z_v}\Big)+\sum_{v\in V}{c_v z_v}+\sum_{C \in \C}{\sum_{e \in \bar P(C)} c_e z_e} \\
     \st \quad & z \in   \MP^{F}(\G). \nonumber
\end{align}
In~\cite{dPKha18SIOPT}, the authors prove that $\MP^F(\G) =\MP(\G)$ if and only if $\G$ is a $\gamma$-acyclic hypergraph. Note that $\gamma$-acyclic hypergraphs represent a significant generalization of Berge-acyclic hypergraphs~\cite{fag83}. While a UGM hypergraph is not
$\gamma$-acyclic in general, as we show in our numerical experiments, the flower LP is significantly stronger than the standard LP.

\subsection{The running intersection relaxation}
In~\cite{dPKha21MOR}, the authors introduce running intersection inequalities, a family of valid inequalities for the multilinear polytope, which strengthens the flower relaxation (see~\cite{dPKhaSah20MPC} for a detailed computational study). Running intersection inequalities were later generalized in~\cite{dPKha23mMPA} and in the following we use this more general definition. To define these inequalities, we first introduce the notion of running intersection property~\cite{BeFaMaYa83}.
A set $F$ of subsets of a finite set $V$ has the \emph{running intersection property} if there exists an ordering $p_1, p_2, \ldots, p_m$ of the sets in $F$ such that
\begin{equation}
\label{ripeq}
\text{for each $k = 2, \dots,m$, there exists $j < k$ such that $p_k \cap \Big(\bigcup_{i < k}{p_i}\Big) \subseteq p_j$.}
\end{equation}
We refer to an ordering $p_1, p_2, \ldots, p_m$ satisfying~\eqref{ripeq} as a \emph{running intersection ordering} of $F$.
Each running intersection ordering $p_1, p_2, \ldots, p_m$ of $F$ induces a collection of sets
\begin{equation}
\label{ripeqsets}
N(p_1) := \emptyset, \qquad N(p_k) := p_k \cap \Big(\bigcup_{i < k}{p_i}\Big) \text{ for } k=2,\dots,m.
\end{equation}
We are now ready to define running intersection inequalities.
Let $e_0$, $e_k$, $k \in T$, be edges of $\G$ such that
\begin{itemize}
\item[(i)] $|e_0 \cap e_k| \geq 2$ for all $k \in T$,
\item[(ii)] $e_0 \cap e_k \not\subseteq e_0 \cap e_{k'}$ for any $k\neq  k' \in T$,
\item[(iii)] the set $\tilde E := \{e_0 \cap e_k : k \in T\}$
has the running intersection property.
\end{itemize}
Consider a running intersection ordering of $\tilde E$ with the sets $N(e_0 \cap e_k)$, for all $k \in T$, as defined in~\eqref{ripeqsets}.
For each $k \in T$, let $w_k \subseteq N(e_0 \cap e_k)$ such that $w_k \in \{\emptyset\} \cup V \cup \E$.
We define a~\emph{running intersection inequality} centered at $e_0$ with neighbors $e_k$, $k \in T$ as:
\begin{equation}
\label{eq: rie}
- \sum_{k \in T} z_{w_k}
+ \sum_{v \in e_0 \setminus \bigcup_{k \in T} e_k} z_v
+ \sum_{k \in T}{z_{e_k}}
- z_{e_0} \leq \omega-1,
\end{equation}
where we define $z_{\emptyset} := 0$, and
$$\omega := \Big|e_0 \setminus \bigcup_{k \in T} e_k \Big|+\Big|\Big\{k \in T : N(e_0 \cap e_k) = \emptyset \Big\}\Big|.$$
Notice that by letting $w_k = \emptyset$ for all $k \in T$, the running intersection inequality~\eqref{eq: rie} simplifies to the flower inequality~\eqref{flowerIneq}.
Consider $e_0, e_k$, $k \in T$ such that $\tilde E$ has the running intersection property and $N(e_0 \cap e_k) \neq \emptyset$ for some $k \in T$. Then any running intersection inequality centered at $e_0$  with neighbors $e_k$, $k \in T$ satisfying $w_k \neq \emptyset$ for some $k \in T$ together with $z_{w_k} \leq 1$ imply the flower inequality centered at $e_0$  with neighbors $e_k$, $k \in T$. However, in general flower inequalities are not implied by running intersection inequalities, since for flower inequalities we do not require the set $\tilde E$ to have the running intersection property.

We then define the~\emph{running intersection relaxation} of $S(\G)$, denoted by $\MP^{\rm RI}(\G)$, as the polytope obtained by adding to the flower relaxation, all running intersection inequalities of $S(\G)$. As in the case of flower inequalities, for a UGM hypergraph, we can establish the redundancy of a large number of running intersection inequalities:

\begin{lemma}\label{lem:redrI}
    Let $\G=(V,\E)$ with $\E=\cup_{C \in \C}{\bar P(C)}$ be a UGM hypergraph and consider the running intersection relaxation $\MP^{\rm RI}(\G)$ of $S(\G)$. Denote by
    $\MP^{\rm RI'}(\G)$ the polytope obtained by adding all running intersection inequalities~\eqref{eq: rie} satisfying condition~\eqref{nonred}
    to the flower relaxation. Then $\MP^{\rm RI}(\G) = \MP^{\rm RI'}(\G)$.
\end{lemma}
\begin{proof}
    The proof follows from a similar line of arguments to those in the proof of Lemma~\ref{lem:redfl}.
\end{proof}

We then define our next LP relaxation, which we refer to as the \emph{running LP}:
\begin{align}\label{runLP}\tag{runLP}
    \max \quad &  \alpha \Big(\sum_{v\in V: y_v=1}{z_v}-\sum_{v\in V: y_v=0}{z_v}\Big)+\sum_{v\in V}{c_v z_v}+\sum_{C \in \C}{\sum_{e \in \bar P(C)} c_e z_e} \\
     \st \quad & z \in   \MP^{\rm RI}(\G). \nonumber
\end{align}
In~\cite{dPKhaSah20MPC}, the authors prove that for hypergraphs with fixed rank, the separation problem over running intersection inequalities can be solved in polynomial time. Our computational results indicate that the running LP is significantly stronger than the flower LP. However, the added strength often comes at the cost of a rather significant increase in CPU time.

\subsection{The clique relaxation}
\label{sec:cliqueLP}

A hypergraph $\bar G$ with node set $\bar V$ is a \emph{complete hypergraph}, if its edge set consists of all subsets of $\bar V$ of cardinality at least two.
It then follows that a UGM hypergraph $\G=(V,\E)$ with $\E=\cup_{C \in \C}{\bar P(C)}$ can be written as a union of complete hypergraphs $\G = \cup_{C \in \C}{\G_C}$, where $\G_C$ denotes a complete hypergraph with node set $C$. We then define the \emph{clique relaxation} of the multilinear set, denoted by $\MP^{\rm cl}(\G)$, as the polytope obtained by intersecting all multilinear polytopes of complete hypergraphs; \ie $\MP(\G_C)$ for all $C \in \C$. An explicit description for the multilinear polytope of a complete hypergraph can be obtained using Reformulation Linearization Technique (RLT)~\cite{SheAda90}.
For completeness, we present this description next.

\begin{prop}[Theorem~2 in~\cite{SheAda90}]\label{prop:RLT}
 Let $\G_C$ be a complete hypergraph with node set $C$. Then the multilinear polytope $\MP(\G_C)$ is given by
\begin{equation}
\psi_U(z_C) \geq 0 \quad \forall U \subseteq C\label{eq:rlt},
\end{equation}
where
\begin{equation}\label{defRLT}
    \psi_U(z_C):=\sum_{\substack{W \subseteq C \cap U: \\ |W| \;{\rm even}}}{z_{(C\setminus U)\cup W}}-\sum_{\substack{W \subseteq C \cap U: \\ |W|\; {\rm odd}}}{z_{(C\setminus U)\cup W}},
\end{equation}
and we define $z_{\emptyset} := 1$.
\end{prop}
By Proposition~\ref{prop:RLT}, the clique relaxation $\MP^{\rm cl}(\G)$ consist of $\sum_{C \in \C}{2^{|C|}}$ inequalities. Hence, this relaxation is computationally tractable only if the rank of the UGM hypergraph is small; a property that is present in all relevant applications. We now define our next LP relaxation which we refer to as the \emph{clique LP}:
\begin{align}\label{cliqueLP}\tag{cliqueLP}
    \max \quad &  \alpha \Big(\sum_{v\in V: y_v=1}{z_v}-\sum_{v\in V: y_v=0}{z_v}\Big)+\sum_{v\in V}{c_v z_v}+\sum_{C \in \C}{\sum_{e \in \bar P(C)} c_e z_e} \\
     \st \quad & z \in   \MP^{\rm cl}(\G). \nonumber
\end{align}
In Sections~\ref{sec:images} and~\ref{sec:decoding} we show that the clique LP returns a binary solution in many cases of practical interest. We next present a theoretical justification of this fact. That is, we show that all inequalities defining facets of the clique relaxation $\MP^{\rm cl}(\G)$ are facet-defining for the multilinear polytope $\MP(\G)$ as well. To this end, we make use of a zero-lifting operation for the multilinear polytope that was introduced in~\cite{dPKha17MOR}. Let $\G=(V,\E)$ be a hypergraph. Then the hypergraph $\G^{\prime}=(V^{\prime}, \E^{\prime})$ is a \emph{partial hypergraph} of $\G$ if $V^{\prime} \subseteq V$ and
$\E^{\prime} \subseteq \E$.
The following lemma provides a sufficient condition under which a facet-defining inequality for $\MP(\G')$ is also facet-defining for $\MP(\G)$.

\begin{lemma}[Corollary~4 in~\cite{dPKha17MOR}]\label{lem:lifting}
Let the complete hypergraph $\G'=(V',\E')$ be a partial hypergraph of $\G=(V, \E)$. Then all facet-defining inequalities for $\MP(\G')$ are facet-defining for $\MP(\G)$ if and only if there exists no edge $e \in \E$ such that $e \supset V'$.
\end{lemma}

The following result establishes the strength of the clique relaxation:

\begin{prop}
    Let $\G = \cup_{C \in \C}{\G_C}$ be a UGM hypergraph where $\C$ denotes the set of maximal clique of the binary UGM. Then for any $C \in \C$, all facet-defining inequalities of $\MP(\G_C)$ are facet defining for $\MP(\G)$ as well.
\end{prop}
\begin{proof}
Since $C \in \C$ is a maximal clique of the UGM, by definition, there exists no edge $e \in \E$ that strictly contains $C$. Hence, the assumptions of Lemma~\ref{lem:lifting}
are satisfied and the result follows.
\end{proof}

By construction, for a general UGM hypergraph $\G$ we have $\MP^{\rm RI}(\G) \subset \MP^F(\G) \subset \MP^{\rm \LP}(\G)$.
The next result indicates that the clique relaxation $\MP^{\rm cl}(\G)$ is the strongest relaxation introduced so far.

\begin{prop}\label{prop:compare}
    Let $\G = \cup_{C \in \C}{\G_C}$ be a UGM hypergraph of rank $r \geq 3$, where $\C$ denotes the set of maximal clique of the binary UGM. Then $\MP^{\rm cl}(\G) \subset
    \MP^{\rm RI}(\G)$.
\end{prop}
\begin{proof}
    Consider any inequality in the description of $\MP^{\rm RI}(\G)$. Then from the definition of the standard linearization together with Lemmas~\ref{lem:redfl} and~\ref{lem:redrI}, it follows that this inequality is also a valid inequality for the multilinear polytope $\MP(\G_C)$ for some $C \in \C$, and hence is implied by inequalities defining $\MP^{\rm cl}(\G)$. Moreover, $\MP^{\rm cl}(\G)$ is strictly contained in $\MP^{\rm RI}(\G)$ since for example its facet-defining inequality~\eqref{eq:rlt} with $U=C$ is given by
    $$
    \sum_{\substack{W \subseteq C: \\ |W| \;{\rm odd}}}{z_{W}}-\sum_{\substack{W \subseteq C: \\ |W|\; {\rm even}}}{z_{W}} \leq 0,
    $$
    which is not present in  $\MP^{\rm RI}(\G)$, if $|C| \geq 3$. To see this, note that in the above inequality, the cardinality of the set of variables with negative coefficients, \ie $\{W \subseteq C: |W| \;{\rm even}, |W| \geq 2\}$ is lower bounded by $\binom{|C|}{2}$ which is larger than one, if $|C| \geq 3$. However, in a running intersection inequality~\eqref{eq: rie}, there is only one variable $z_e$, for some $e \in \E$ with a negative coefficient.
\end{proof}

Next we provide a sufficient condition under which the clique relaxation coincides with the multilinear polytope. To this we make use of a sufficient condition for decomposability of multilinear sets~\cite{dPKha18MPA}. Let $\G=(V,\E)$ be a hypergraph and let $\G'=(V',\E')$ be a partial hypergraph of $\G$. Then $\G'$ is a \emph{section hypergraph} of $\G$ \emph{induced by} $V'$, if $\E'=\{e \in \E: e \subseteq V'\}$. Given hypergraphs $\G_1 = (V_1,\E_1)$ and
$\G_2 = (V_2,\E_2)$, we denote by $\G_1 \cap \G_2$ the hypergraph $(V_1 \cap V_2, \E_1 \cap \E_2)$, and we denote
by $\G_1 \cup \G_2$, the hypergraph $(V_1\cup V_2,\E_1 \cup \E_2)$. Finally, we say that the multilinear set $S(\G)$ is \emph{decomposable} into $S(\G_1)$ and $S(\G_2)$ if the system comprised of the description of $\MP(\G_1)$ and the description of $\MP(\G_2)$, is the description of $\MP(\G)$.
\begin{theorem} [Theorem~1 in~\cite{dPKha18MPA}]
\label{decomposability}
Let $\G$ be a hypergraph, and let $\G_1, \G_2$ be section hypergraphs of $\G$ such that $\G_1 \cup \G_2=\G$ and $\G_1 \cap \G_2$ is a complete hypergraph. Then the set $S(\G)$ is decomposable into $S(\G_1)$ and $S(\G_2)$.
\end{theorem}

\begin{prop}\label{prop:charac}
     Let $\G = \cup_{C \in \C}{\G_C}$ be a UGM hypergraph where $\C$ denotes the set of maximal clique of the binary UGM and $\G_C$ is a complete hypergraph with node set $C$. If $\C$ has the running intersection property, then $\MP^{\rm cl}(\G)=\MP(\G)$ .
\end{prop}
\begin{proof}
    Denote by $C_1, C_2, \cdots, C_m$ a running intersection ordering of $\C$. Let $\G_i$ denote the complete hypergraph with node set $C_i$ for all $i \in \{1,\cdots,m\}$. Define $\G'_m=\cup_{i=1}^{m-1}{\G_{i}}$.
    Then by definition of the running intersection ordering, $\G_m \cap \G'_m$ is a section hypergraph of $\G_i$
    for some $i \in \{1,\cdots,m-1\}$ and hence is a complete hypergraph. Hence, by Theorem~\ref{decomposability}
    the set $S(\G)$ is decomposable into $S(\G_m)$ and $S(\G'_m)$. By a recursive application of this argument $m$ times we conclude that $S(\G)$ is decomposable into multilinear sets $S(\G_i)$, $i \in \{1,\cdots,m\}$, implying $\MP^{\rm cl}(\G)=\MP(\G)$.
\end{proof}

\subsection{The multi-clique relaxation}

By Proposition~\ref{prop:compare}, the clique relaxation is the strongest LP relaxation introduced so far. Yet by Proposition~\ref{prop:charac}, this LP is guaranteed to solve the original problem if the set of cliques $\C$ has the running intersection property; a property that is often~\emph{not} present in applications. In this section we propose stronger LP relaxations for Problem~\ref{map6}
by constructing the multilinear polytope of a UGM hypergraph containing multiple cliques that do not have the running intersection property. More precisely, we consider a special structure, that we refer to as the \emph{lifted cycle of cliques}, and obtain the multilinear polytope using disjunctive programming~\cite{balas1998disjunctive}. This structure appears in applications such as image restoration and decoding problems.

 Consider the set of maximal cliques $\C:= \{C_1, C_2, \ldots, C_m\}$, where $m \geq 3$ and $|C_i| \geq 3$ for all $i \in [m]:=\{1,\cdots,m\}$. We say that $\C$ is a \emph{lifted cycle of cliques} if $C_i \cap C_{i+1}=\{v_i, \bar v\}$ for all $i \in [m]$,
 where $v_i \neq v_j$ for any $i \neq j \in [m]$ and where we define $C_{m+1} := C_1$.
It is simple to check if $\C$ is a lifted cycle of cliques, then it does not have the running intersection property. Figure~\ref{fig:cyclesofcliques} illustrates examples of cycles of cliques that appear in applications.
The objective of this section is to characterize $\MP(\G_{\C})$, where $\C$
is a lifted cycle of cliques. To this end, next we introduce a lifting operation for the multilinear polytope which is the key to our characterization.
For notational simplicity, for a node $v$, we use the notations $z_v$ (resp. $z_{v\cup e}$ for some $e \in \E$) and $z_{\{v\}}$ (resp. $z_{\{v\}\cup e}$ for some $e \in \E$), interchangeably.

\begin{figure}
    \centering
  \begin{subfigure}[t]{.45\textwidth}
    \centering
    \begin{tikzpicture}[scale=0.8]
      \tikzset{every node/.style={circle, draw=black, fill=blue!20, inner sep=2pt}}
      \node[] (1) at (0, 0) {};
      \node[] (2) at (-0.75, -1.275) {};
      \node[] (3) at (-0.75, 1.275) {};
      \node[] (4) at (1.5, 0) {};
      \node[] (5) at (-1.5, 0) {};
      \node[] (6) at (0.75, -1.275) {};
      \node[] (7) at (0.75, 1.275) {};

      \draw[line width=2pt, draw=lightgreen] (-0.75,0) ellipse (1.0cm and 1.6cm);
      \draw[line width=2pt, draw=lightgreen, rotate=-30] (0.075,0.75) ellipse (1.875cm and 0.975cm);
      \draw[line width=2pt, draw=lightgreen, rotate=30] (0.075,-0.75) ellipse (1.875cm and 0.975cm);
    \end{tikzpicture}
    \caption{}
  \end{subfigure}%
  \hspace{0.5cm}
    \begin{subfigure}[t]{.45\textwidth}
    \centering
        \begin{tikzpicture}[scale=1.1]
          \foreach \i in {0,1,2} {
            \foreach \j in {0,1,2} {
              \node[circle, draw = black, fill=blue!20, inner sep=2pt] (n\i\j) at (\i,\j) {};
            }
          }
          \foreach \i in {0,1} {
            \foreach \j in {0,1} {
              \draw[rounded corners=28pt, line width=2pt, draw=lightgreen] (\i-0.4,\j-0.4) rectangle (\the\numexpr\i+1\relax+0.4,\the\numexpr\j+1\relax+0.4);
            }
          }
        \end{tikzpicture}
        \caption{}
    \end{subfigure}
   \caption{Examples of lifted cycle of cliques of length $m=3$ and $m=4$.}
\label{fig:cyclesofcliques}
\end{figure}
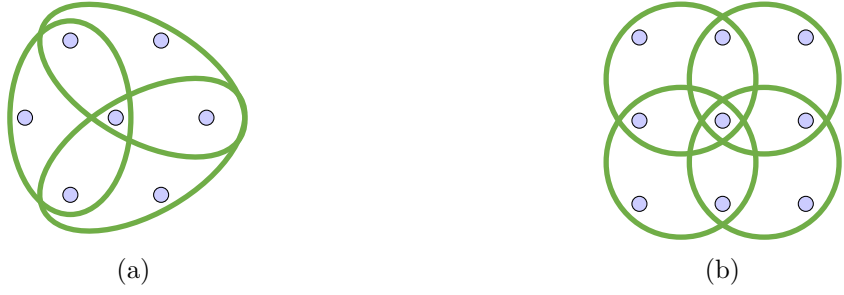

\begin{prop}
\label{prop_fix}
Let $\G'=(V',\E')$ be a hypergraph and let $\bar v \notin V'$. Define the hypergraph $\G$
with the node set $V:= V' \cup\{\bar v\}$ and the edge set $\E:=\{e\cup\{\bar v\}: e \in \E'\}$. Suppose that $\MP(\G')$ is defined by:
\begin{equation}\label{base}
\sum_{v\in V^{'}}a^i_{v}z_{v}+\sum_{e\in \E^{'}}a^i_{e}z_{e}\leq \alpha_i, \quad \forall i \in I.
\end{equation}
Then $\MP(\G)$ is defined by the following inequalities:
\begin{align}
\label{eq:substitution}
\begin{split}
& 0 \leq z_{\bar v} \leq 1\\
&\sum_{v\in V'} {a^i_v z_{\{v, \bar v\}}}+\sum_{e\in \E'} {a^i_e z_{e \cup \{\bar v\}}}\leq \alpha_i z_{\bar v}, \quad \forall i \in I,\\
&\sum_{v\in V'} {a^i_{v}(z_{v}-z_{\{v,\bar v\}})}+\sum_{e\in \E^{'}}  {a^i_{e}(z_{e}-z_{e\cup\{\bar v\}})}\leq \alpha_i(1-z_{\bar v}), \quad \forall i \in I.
\end{split}
\end{align}
\end{prop}
\begin{proof}
Denote by $\MP(\G_0)$ (resp. $\MP(\G_1)$) the face of $\MP(\G)$ with $z_{\bar v} = 0$ (resp. $z_{\bar v} = 1$).
We then have:
$$\MP(\G) = \conv(\MP(\G_0) \cup \MP(\G_1)).$$
Since $\MP(\G')$ is defined by inequalities~\eqref{base}, it follows that $\MP(\G_0)$ is given by inequalities~\eqref{base} together with
$z_{\bar v} = 0$, $z_{p \cup \{\bar v\}} = 0$ for all $p \in V'\cup \E'$,
while $\MP(\G_1)$ is given by inequalities~\eqref{base} together with:
$z_{\bar v} = 1$, $z_{p \cup \{\bar v\}} = z_p$, for all $p \in V'\cup \E'$.
Using Balas’ formulation for the union of polytopes~\cite{balas1998disjunctive}, we deduce that $\MP(\G)$ is the projection onto the space of the $z$ variables of the polytope defined by the following system:
\begin{align}
&\sum_{v\in V'}{a^i_{v}z_{v}^{1}}+\sum_{e\in \E'}{a^i_{e}z_{e}^{1}}\leq \alpha_i\lambda,\quad \forall i \in I\label{ineq1}\\
& z^1_{\bar v} = 0, \quad z^1_{p \cup \{\bar v\}} = 0, \quad \forall p \in V'\cup \E' \nonumber\\
&\sum_{v\in V'}{a^i_{v}z_{v}^{2}}+\sum_{e\in \E'}{a^i_{e}z_{e}^{2}}\leq \alpha_i(1-\lambda),\quad \forall i \in I\label{ineq2}\\
&z^2_{\bar v} = 1-\lambda, \quad z^2_{p \cup \{\bar v\}} = z^2_p, \quad  p \in V'\cup \E' \nonumber\\
& z_p = z^1_p + z^2_p  \quad \forall p \in V \cup \E \nonumber\\
&0\leq \lambda\leq 1. \label{ineq3}
\end{align}
To complete the proof, we should project out variables $z^1, z^2, \lambda$ from the above system.
From $z_{\bar v} = z^1_{\bar v} + z^2_{\bar v}$, $z^1_{\bar v} = 0$, and $z^2_{\bar v} = 1-\lambda$, we get
\begin{equation}\label{firstp}
\lambda=1-z_{\bar v}.
\end{equation}
For each $p \in V' \cup \E'$, we have $z_{p \cup \{\bar v\}} = z^1_{p \cup \{\bar v\}} + z^2_{p \cup \{\bar v\}} = 0 + z^2_p$, implying that
\begin{equation}\label{secondp}
z^2_p = z_{p \cup \{\bar v\}}, \quad \forall p \in V' \cup \E'.
\end{equation}
For each $p \in V' \cup \E'$ we have $z_p = z^1_p + z^2_p$; combining this with~\eqref{secondp}, we obtain
\begin{equation}\label{thirdp}
z^1_p = z_p-z_{p \cup \{\bar v\}}, \quad \forall p \in V' \cup \E'.
\end{equation}
Substituting~\eqref{firstp}-\eqref{thirdp} in~\eqref{ineq1}-\eqref{ineq3}, we obtain system~\eqref{eq:substitution} and this completes the proof.
\end{proof}

Odd-cycle inequalities are a well-known class of valid inequalities for the Boolean quadric polytope~\cite{Pad89}. These inequalities play an important role in characterizing the multilinear polytope of a lifted cycle of cliques. We define these inequalities next;
let $G=(V,E)$ be a graph. Padberg~\cite{Pad89} introduced the Boolean quadric polytope of $G$, denoted by ${\rm BQP}(G)$, as follows:
$$
{\rm BQP}(G) := \conv\Big\{z\in \{0,1\}^{V\cup E}: z_e = z_u z_v, \; \forall \{u,v\} \in E \Big\}.
$$
Let $K$ be a cycle of $G$. We denote by $V(K)$ the nodes of the cycle. Let $D \subseteq K$ such that $|D|$ is odd.
Define $V_1(D) := (\cup_{e \in D}{e}) \setminus (\cup_{e \in K \setminus D}{e})$ and $V_2(D) :=V(K) \setminus (\cup_{e \in D}{e})$. Then an \emph{odd-cycle inequality} for ${\rm BQP}(G)$ is given by:
\begin{equation}\label{oddCycle}
\sum_{v \in V_1(D)}{z_v}-\sum_{v \in V_2(D)}{z_v}-\sum_{e \in D}{z_e}+\sum_{K \setminus D}{z_e} \leq \Big\lfloor\frac{|D|}{2}\Big\rfloor.
\end{equation}
Padberg proved that if the graph $G$ consists of a chordless cycle $K$, then the polytope obtained by adding all odd-cycle inequalities of $K$ to the standard linearization coincides with the Boolean quadric polytope ${\rm BQP}(G)$ (see Theorem~9 in~\cite{Pad89}).

\medskip
We are now ready to state the main result of this section.

\begin{prop}\label{ccc}
    Let $\C=\{C_1, \cdots, C_m\}$ be a lifted cycle of cliques
    with $C_i \cap C_{i+1}=\{v_i, \bar v\}$ for $i \in [m]$, where $C_{m+1}:= C_1$. Define the cycle $K :=\{\{v_{1}, v_{i+1}\}, \forall i \in [m]\}$, where $v_{m+1}:= v_1$.
    Denote by $\G_{C_i}$, $i \in [m]$ the complete hypergraph
    with node set $C_i$ and let $\G_{\C} = \cup_{i \in [m]} {\G_{C_i}}$.
    Then $\MP(\G_{\C})$ is defined by the system comprising of inequalities defining $\MP(\G_{C_i})$ for $i \in [m]$, and the following inequalities:
    \begin{align}\label{liftedoddCycle}
& \sum_{v \in V_1(D)}{z_{\{v,\bar v\}}}-\sum_{v \in V_2(D)}{z_{\{v,\bar v\}}}-\sum_{e \in D}{z_{e\cup \{\bar v\}}}+\sum_{K \setminus D}{z_{e \cup \{\bar v\}}} \leq \Big\lfloor\frac{|D|}{2}\Big\rfloor z_{\bar v},\nonumber\\
& \sum_{v \in V_1(D)}{(z_v-z_{\{v,\bar v\}})}-\sum_{v \in V_2(D)}{(z_v-z_{\{v,\bar v\}})}-\sum_{e \in D}{(z_e-z_{e \cup \{\bar v\}})}+\sum_{K \setminus D}{(z_e-z_{e \cup \{\bar v\}})} \leq \Big\lfloor\frac{|D|}{2}\Big\rfloor (1-z_{\bar v}),\nonumber\\
& \qquad \qquad \qquad \qquad \qquad \forall D \subseteq K: |D| \; {\rm is \; odd},
\end{align}
where for each $D$ we define $V_1(D) := (\cup_{e \in D}{e}) \setminus (\cup_{e \in K \setminus D}{e})$ and $V_2(D) :=V(K) \setminus (\cup_{e \in D}{e})$, and as before $V(K)$ denotes the node set of the cycle $K$.
\end{prop}

\begin{proof}
    Define $C'_i := C_i \setminus \{\bar v\}$ for all $i \in [m]$. Denote by $\G_{C'_i}$, the complete hypergraph
    with node set $C'_i$ for $i \in [m]$ and define $\G_{\C'} = \cup_{i\in [m]}{\G_{C'_i}}$. By Proposition~\ref{prop_fix}, to characterize $\MP(\G_{\C})$ it suffices to characterize $\MP(\G_{\C'})$. We next obtain the explicit description for $\MP(\G_{\C'})$.

    The hypergraph $\G_{C'_1}$ is the section hypergraph of $\G_{\C'}$ induced by $C'_1$. Denote by $\G_2$ the section hypergraph $\G_{\C'}$ induced by $\cup_{i =2}^m{C'_i}$. Notice that $\G_2 = \cup_{i=2}^m {\G_{C'_i}} \cup F_m$, where $F_m$ is the graph with node set $\{v_1,v_m\}$ and edge set $\{\{v_1, v_m\}\}$.
    We then have $\G_{\C'} = \G_{C'_1} \cup \G_2$ and
    $\G_{C'_1} \cap \G_2 = \{v_1, v_m\}$. Therefore, by Theorem~\ref{decomposability} the multilinear set $S(\G_{\C'})$ decomposes into the multilinear sets
    $S(\G_{C'_1})$ and $S(\G_2)$. Next consider the hypergraph $\G_2$; denote by $\G_3$ the section hypergraph of $\G_2$ induced by $(\cup_{i =3}^m{C'_i}) \cup \{v_1\}$. Notice that $\G_3= \cup_{i=3}^m {\G_{C'_i}} \cup F_m \cup F_1$, where $F_1$ is the graph with node set $\{v_1,v_2\}$ and edge set $\{\{v_1, v_2\}\}$.
    We then have $\G_2 = \G_{C'_2} \cup \G_3$ and
    $\G_{C'_2} \cap \G_3 = \{v_1, v_2\}$. Therefore, by Theorem~\ref{decomposability}, $S(\G_2)$ decomposes into
    $S(\G_{C'_2})$ and $S(\G_3)$.
    Applying this argument recursively, we conclude that $S(\G_{\C'})$ decomposes into  $S(\G_{C'_i})$ for all $i \in [m]$ and $S(K)$, where the node set of $K$ is given by $\{v_1, \cdots, v_m\}$ and the edge set of $K$ is given by $\{\{v_i,v_{i+1}\}, \forall i \in [m]\}$, where $v_{m+1}:=v_1$. First notice that the multilinear polytope $\MP(\G_{C'_i})$, $i \in [m]$ is given by Proposition~\ref{prop:RLT} as $\G_{C'_i}$ is a complete hypergraph. Moreover, the multilinear polytope $\MP(K)$
    is given by Theorem~9 in~\cite{Pad89} as $K$ is a chordless cycle; \ie $\MP(K)$ consists of odd-cycle inequalities~\eqref{oddCycle} together with inequalities defining $\MP(F_i)$, $i \in [m]$ where the node set and the edge set of $F_i$ are given by $\{v_i, v_{i+1}\}$
    and $\{\{v_i, v_{i+1}\}\}$, where $v_{m+1} := v_1$. Since
    $F_i$ is a section hypergraph of $\C'_i$ for all $i \in [m]$, we deduce that the inequalities defining $\MP(F_i)$ are implied by inequalities defining $\MP(\G_{C'_i})$.    Therefore, $\MP(\G_{\C'})$ is defined by inequalities defining  $\MP(\G_{C'_i})$, $i \in [m]$ together with odd-cycle inequalities for $K$. Therefore, by Proposition~\ref{prop_fix}, $\MP(\G_{\C})$ is defined by inequalities defining $\MP(\G_{C_i})$, $i \in [m]$ together with inequalities~\eqref{liftedoddCycle}.
\end{proof}

Henceforth, we refer to inequalities~\eqref{liftedoddCycle} as the \emph{lifted odd-cycle inequalities}.
We then define the \emph{multi-clique relaxation} of the multilinear set, denoted by $\MP^{\rm Mcl}(\G)$, as the polytope obtained by adding all lifted odd-cycle inequalities~\eqref{liftedoddCycle} corresponding to cycles of cliques of length at most $m \leq M$ to the clique relaxation $\MP^{\rm cl}(\G)$. We now define our final LP relaxation  which we refer to as the \emph{multi-clique LP}:
\begin{align}\label{McliqueLP}\tag{McliqueLP}
    \max \quad &  \alpha \Big(\sum_{v\in V: y_v=1}{z_v}-\sum_{v\in V: y_v=0}{z_v}\Big)+\sum_{v\in V}{c_v z_v}+\sum_{C \in \C}{\sum_{e \in \bar P(C)} c_e z_e} \\
     \st \quad & z \in   \MP^{\rm Mcl}(\G). \nonumber
\end{align}
To control the computational cost of solving Problem~\ref{McliqueLP}, in all our numerical experiments we will set $M = 4$. That is, lifted odd-cycle inequalities are considered for cycles of cliques of length $m \in \{3,4\}$.

In~\cite{del2021chvatal,dPWal23MPB}, the authors introduce odd $\beta$-cycle inequalities, a class of valid inequalities for the multilinear polytope. Let $\C$ be a lifted cycle of cliques of length $m$ and let $\G_{\C}$ be the corresponding UGM hypergraph. Then it can be checked that there exist odd $\beta$-cycle inequalities of $\MP(\G_{\C})$ that are not implied by the clique relaxation $\MP^{\rm cl}(\G_{\C})$. However, by Proposition~\ref{ccc}, thanks to inequalities~\eqref{liftedoddCycle}, odd $\beta$-cycle inequalities are implied by the multi-clique relaxation $\MP^{\rm mcl}(\G_{\C})$. For a general UGM hypergraph $\G$, there may exist odd $\beta$-cycle inequalities are not implied by any multi-clique relaxation. However, as we demonstrate in Sections~\ref{sec:images} an~\ref{sec:decoding}, for inference in binary UGMs, the clique LP and the multi-clique LP are often sharp and hence we do not explore other valid inequalities such as odd $\beta$-cycle inequalities.

\section{First application: image restoration}
\label{sec:images}

Images are often degraded during the data acquisition process. The degradation may involve blurring, information loss due to sampling, quantization effects, and various sources of noise.
Image restoration, a popular application in computer vision, aims at recovering the original image from degraded data.  UGMs are a popular tool for modeling the prior information such as \emph{smoothness} in image restoration problems. In this framework, there is a node in the graph corresponding to each pixel in the image and the edges of the graph are chosen to enforce local smoothness conditions.
The majority of the literature on solving the image restoration problem has focused on first-order UGMs, also known as, pairwise models; \ie $|C| = 2$ for all $C \in \C$. The most popular pairwise model for image restoration is the four-nearest neighbors model (see Figure~\ref{fig:cliques} for an illustration)~\cite{willsky02,meltzer05,szeliski08}. While it has been long recognized that higher-order UGMs are better suited for capturing properties of image priors, the complexity of solving Problem~\eqref{map6} has limited their use in practice. Almost all existing literature on higher-order binary UGMs tackles the MAP inference problem by reducing it to a binary quadratic optimization optimization to benefit from efficient optimization algorithms available for that problem class~\cite{ali2008optimizing,ishikawa2009higher,rother2009minimizing,ishikawa2010transformation,ishikawa2014higher,fix2014primal}.

In this paper, we limit our attention to black and white images. An image is a rectangle
consisting of $l \times h$ pixels and it is modeled as a matrix of the same dimension where each element represents a pixel which takes value $0$ or $1$.  In computer vision applications, the cliques of UGMs are often $m \times n$ patches. For example, in a third-order UGM consisting of $2\times2$ patches, each clique consists of four pixels indexed by $(i,j),(i,j+1), (i+1,j), (i+1, j+1)$ for all $i \in [l-1]$ and $j \in [h-1]$ (see Figure~\ref{fig:cliques} for an illustration).
Throughout this section, we make use of this popular model for cliques of the UGM. For the multi-clique LP, we consider cycles of cliques of length four, which is the minimum length for these problems (see Figure~\ref{fig:cyclesofcliques}).

In order to determine the parameters of clique potentials, \ie $c_p$, $p \in V \cup \E$ as defined in Problem~\eqref{map6}, we make use of \emph{pattern-based potentials} introduced in ~\cite{komodakis2009beyond,crama2017class}. Consider a $2 \times 2$ patch in a black and white image;
by symmetry, we can divide all different pixel configurations into four groups (see Table~\ref{table:patterns}).  Letting $\varphi(z):=\sum_{e \in P(C)} c_e \prod_{v \in e} z_v$ for any $C \in \C$,
we then assign the same potential value $\varphi_i$ to all configurations in the $i$th group, essentially stating that they are equally smooth. Using
the values of $\varphi_1, \cdots, \varphi_4$, we can then compute the coefficients $c_e$, $e \in P(C)$.
Hence, it remains to determine parameters $\varphi_1, \cdots, \varphi_4$ and $\alpha$. In the following, we consider two schemes to determine these parameters.

\begin{figure}
\begin{center}
\begin{subfigure}[b]{0.45\textwidth}
\centering
\begin{tikzpicture}[scale=1.0]
\foreach \i in {0,1,2,3} {
  \foreach \j in {0,1,2,3} {
    \node[circle, draw = black, fill=blue!20, inner sep=2pt] (n\i\j) at (\i,\j) {};
  }
}
\foreach \i in {0,1,2} {
  \foreach \j in {0,1,2,3} {
    \draw[line width=2pt, draw=lightgreen] (n\i\j) -- (n\the\numexpr\i+1\relax\j);
  }
}
\foreach \i in {0,1,2,3} {
  \foreach \j in {0,1,2} {
    \draw[line width=2pt, draw=lightgreen] (n\i\j) -- (n\i\the\numexpr\j+1\relax);
  }
}
\end{tikzpicture}
\caption{first-order UGM}
\end{subfigure}
\hspace{0.5cm}
\begin{subfigure}[b]{0.45\textwidth}
\centering
\begin{tikzpicture}[scale=1.0]
\foreach \i in {0,1,2,3} {
  \foreach \j in {0,1,2,3} {
    \node[circle, draw = black, fill=blue!20, inner sep=2pt] (n\i\j) at (\i,\j) {};
  }
}
\foreach \i in {0,1,2} {
  \foreach \j in {0,1,2} {
    \draw[rounded corners=28pt, line width=2pt, draw=lightgreen] (\i-0.4,\j-0.4) rectangle (\the\numexpr\i+1\relax+0.4,\the\numexpr\j+1\relax+0.4);
  }
}
\end{tikzpicture}
\caption{third order UGM}
\end{subfigure}

    \caption{Illustration of clique configurations in a first-order UGM and a third order UGM for image restoration.}
\label{fig:cliques}
\end{center}

\end{figure}
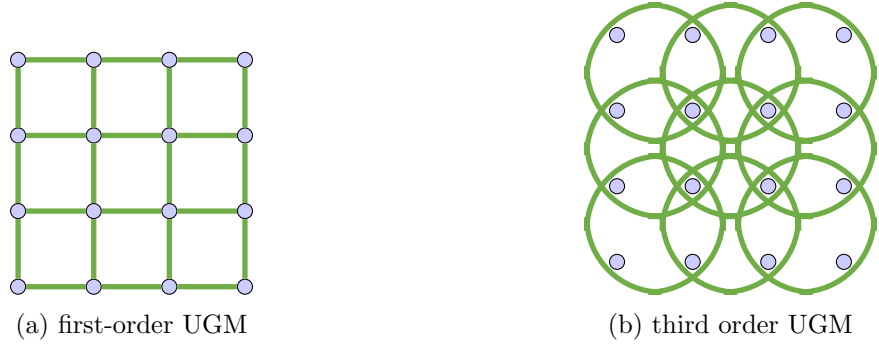

\begin{table}[h]
\caption{Pattern-based potentials for a $2\times 2$ patch in a black and white image}
\label{table:patterns}
\centering
\begin{tabular}{clclclclclclclclc}
\hline
\multicolumn{16}{l}{Variable   assignment}                                                                                                                                       & Potential value         \\ \hline
\multicolumn{2}{c}{\multirow{2}{*}{\begin{tabular}[c]{@{}c@{}}0 0\\      0 0\end{tabular}\quad}} & \multicolumn{2}{c}{\multirow{2}{*}{\begin{tabular}[c]{@{}c@{}}1 1\\      1 1\end{tabular}\quad}} & \multicolumn{2}{c}{\multirow{2}{*}{}}                                                       & \multicolumn{2}{c}{\multirow{2}{*}{}}                                                       & \multicolumn{2}{c}{\multirow{2}{*}{}}                                                       & \multicolumn{2}{c}{\multirow{2}{*}{}}                                                       & \multicolumn{2}{c}{\multirow{2}{*}{}}                                                       & \multicolumn{2}{c}{\multirow{2}{*}{}}                                                       & \multirow{2}{*}{$\varphi_{1}$} \\
\multicolumn{2}{c}{}                                                                        & \multicolumn{2}{c}{}                                                                        & \multicolumn{2}{c}{}                                                                        & \multicolumn{2}{c}{}                                                                        & \multicolumn{2}{c}{}                                                                        & \multicolumn{2}{c}{}                                                                        & \multicolumn{2}{c}{}                                                                        & \multicolumn{2}{c}{}                                                                        &                     \\
\multicolumn{2}{c}{\multirow{2}{*}{\begin{tabular}[c]{@{}c@{}}0 0\\      0 1\end{tabular}\quad}} & \multicolumn{2}{c}{\multirow{2}{*}{\begin{tabular}[c]{@{}c@{}}0 0\\      1 0\end{tabular}\quad}} & \multicolumn{2}{c}{\multirow{2}{*}{\begin{tabular}[c]{@{}c@{}}0 1\\      0 0\end{tabular}\quad}} & \multicolumn{2}{c}{\multirow{2}{*}{\begin{tabular}[c]{@{}c@{}}1 0\\      0 0\end{tabular}\quad}} & \multicolumn{2}{c}{\multirow{2}{*}{\begin{tabular}[c]{@{}c@{}}1 1\\      1 0\end{tabular}\quad}} & \multicolumn{2}{c}{\multirow{2}{*}{\begin{tabular}[c]{@{}c@{}}1 1\\      0 1\end{tabular}\quad}} & \multicolumn{2}{c}{\multirow{2}{*}{\begin{tabular}[c]{@{}c@{}}1 0\\      1 1\end{tabular}\quad}} & \multicolumn{2}{c}{\multirow{2}{*}{\begin{tabular}[c]{@{}c@{}}0 1\\      1 1\end{tabular}\quad}} & \multirow{2}{*}{$\varphi_{2}$} \\
\multicolumn{2}{c}{}                                                                        & \multicolumn{2}{c}{}                                                                        & \multicolumn{2}{c}{}                                                                        & \multicolumn{2}{c}{}                                                                        & \multicolumn{2}{c}{}                                                                        & \multicolumn{2}{c}{}                                                                        & \multicolumn{2}{c}{}                                                                        & \multicolumn{2}{c}{}                                                                        &                     \\
\multicolumn{2}{c}{\multirow{2}{*}{\begin{tabular}[c]{@{}c@{}}1 1\\      0 0\end{tabular}\quad}} & \multicolumn{2}{c}{\multirow{2}{*}{\begin{tabular}[c]{@{}c@{}}0 0\\      1 1\end{tabular}\quad}} & \multicolumn{2}{c}{\multirow{2}{*}{\begin{tabular}[c]{@{}c@{}}1 0\\      1 0\end{tabular}\quad}} & \multicolumn{2}{c}{\multirow{2}{*}{\begin{tabular}[c]{@{}c@{}}0 1\\      0 1\end{tabular}\quad}} & \multicolumn{2}{c}{\multirow{2}{*}{}}                                                       & \multicolumn{2}{c}{\multirow{2}{*}{}}                                                       & \multicolumn{2}{c}{\multirow{2}{*}{}}                                                       & \multicolumn{2}{c}{\multirow{2}{*}{}}                                                       & \multirow{2}{*}{$\varphi_{3}$} \\
\multicolumn{2}{c}{}                                                                        & \multicolumn{2}{c}{}                                                                        & \multicolumn{2}{c}{}                                                                        & \multicolumn{2}{c}{}                                                                        & \multicolumn{2}{c}{}                                                                        & \multicolumn{2}{c}{}                                                                        & \multicolumn{2}{c}{}                                                                        & \multicolumn{2}{c}{}                                                                        &                     \\
\multicolumn{2}{c}{\multirow{2}{*}{\begin{tabular}[c]{@{}c@{}}1 0\\      0 1\end{tabular}\quad}} & \multicolumn{2}{c}{\multirow{2}{*}{\begin{tabular}[c]{@{}c@{}}0 1\\      1 0\end{tabular}\quad}} & \multicolumn{2}{c}{\multirow{2}{*}{}}                                                       & \multicolumn{2}{c}{\multirow{2}{*}{}}                                                       & \multicolumn{2}{c}{\multirow{2}{*}{}}                                                       & \multicolumn{2}{c}{\multirow{2}{*}{}}                                                       & \multicolumn{2}{c}{\multirow{2}{*}{}}                                                       & \multicolumn{2}{c}{\multirow{2}{*}{}}                                                       & \multirow{2}{*}{$\varphi_{4}$} \\
\multicolumn{2}{c}{}                                                                        & \multicolumn{2}{c}{}                                                                        & \multicolumn{2}{c}{}                                                                        & \multicolumn{2}{c}{}                                                                        & \multicolumn{2}{c}{}                                                                        & \multicolumn{2}{c}{}                                                                        & \multicolumn{2}{c}{}                                                                        & \multicolumn{2}{c}{}                                                                        &                     \\ \hline
\end{tabular}
\end{table}

\subsection{Synthetic images}
Our first objective is to compare and contrast various LP relaxations of Problem~\ref{map6} defined in Section~\ref{sec:lprelaxations}. To compare these LPs, we use two metrics: $(i)$ percentage of relative optimality gap defined as $r_g:=\frac{f^*-g^*}{f^*} \times 100$, where $f^*$ denotes the optimal value of Problem~\ref{map6}, while $g^*$ denotes the optimal value of an LP relaxation and $(ii)$ CPU time (seconds). It can be shown that to solve Problem~\ref{map6}, it suffices to add the constraint $z_v \in \{0,1\}$ for all $v \in V$ to Problem~\ref{stdLP}. Henceforth, we refer to the resulting binary integer program as the IP.
We generate random images as described in~\cite{crama2017class}. The authors of~\cite{crama2017class} set the parameters of the inference problem as follows: $\alpha=25$,  $\varphi_1=-10$, $\varphi_2=-20$, $\varphi_3=-30$, and $\varphi_4=-40$. They then
consider three types of ground truth images classified as Top Left Rectangle (TL), Center Rectangle (CEN), and CROSS
(see Figure~\ref{fig:carma_matrix}).
\begin{figure}
\centering
\begin{subfigure}[t]{0.2\textwidth}
\centering
\includegraphics[width=\linewidth]{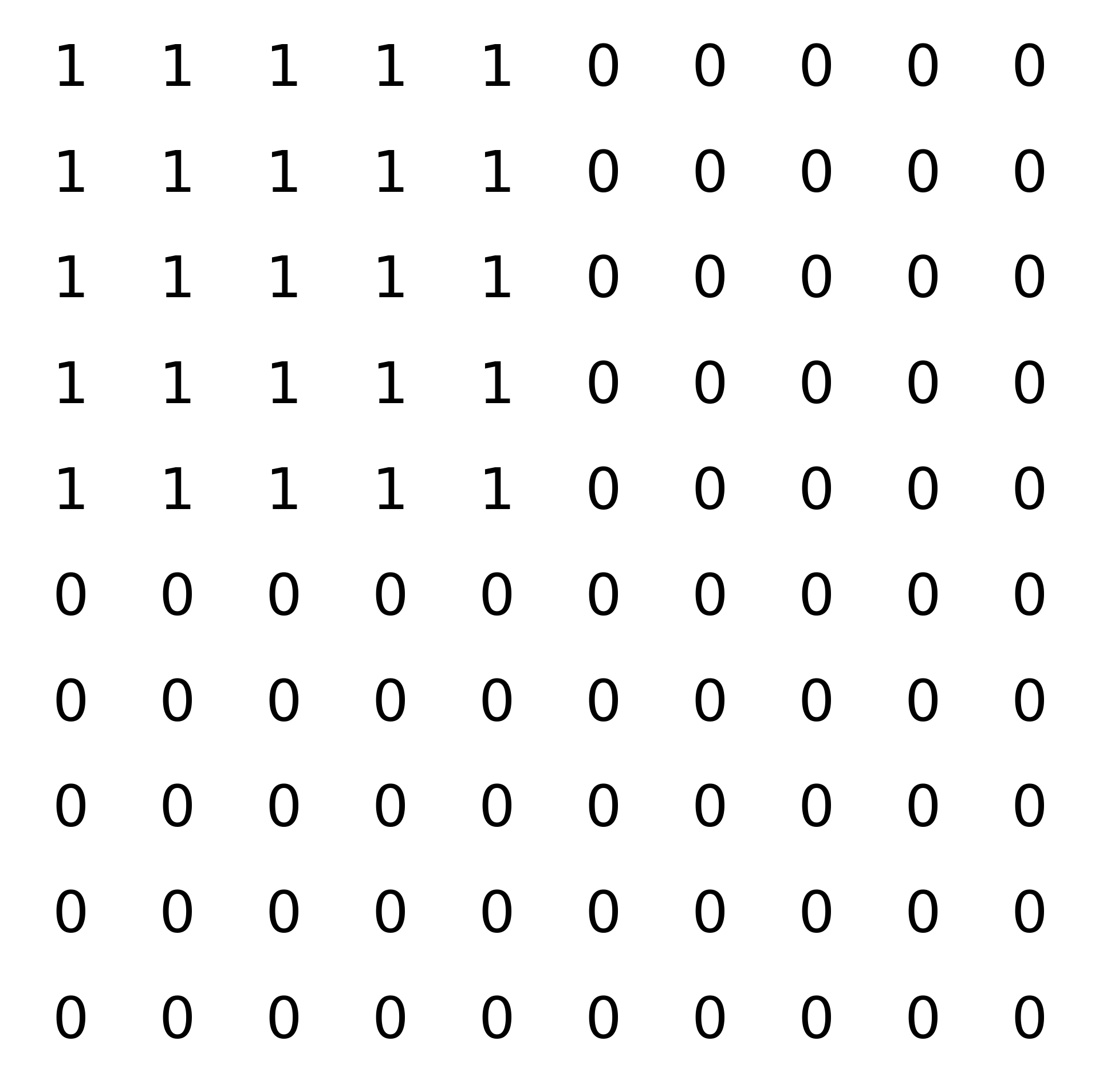}
\caption{TL}
\end{subfigure}
\begin{subfigure}[t]{0.2\textwidth}
\centering
\includegraphics[width=\linewidth]{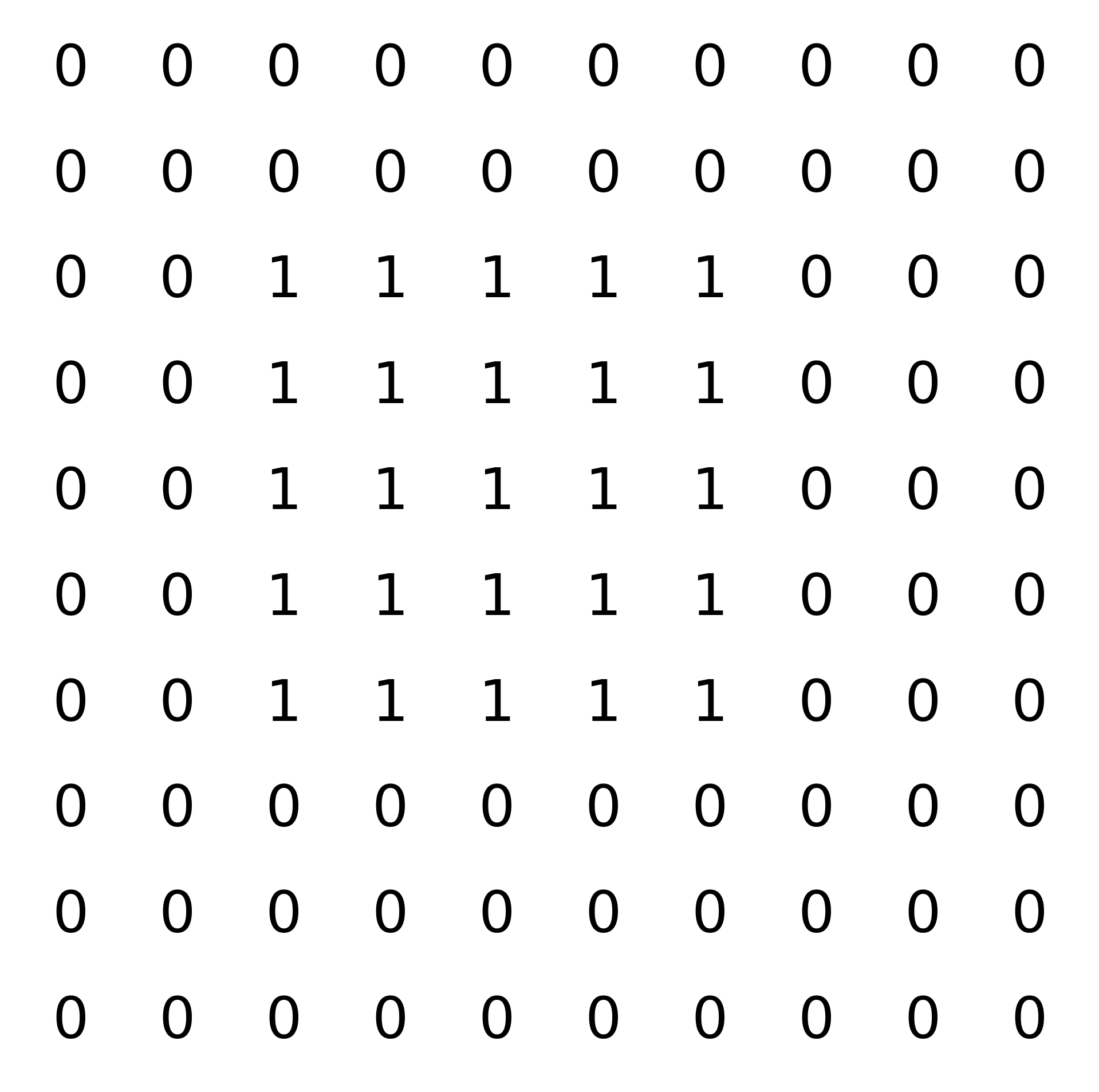}
\caption{CEN}
\end{subfigure}
\begin{subfigure}[t]{0.2\textwidth}
\centering
\includegraphics[width=\linewidth]{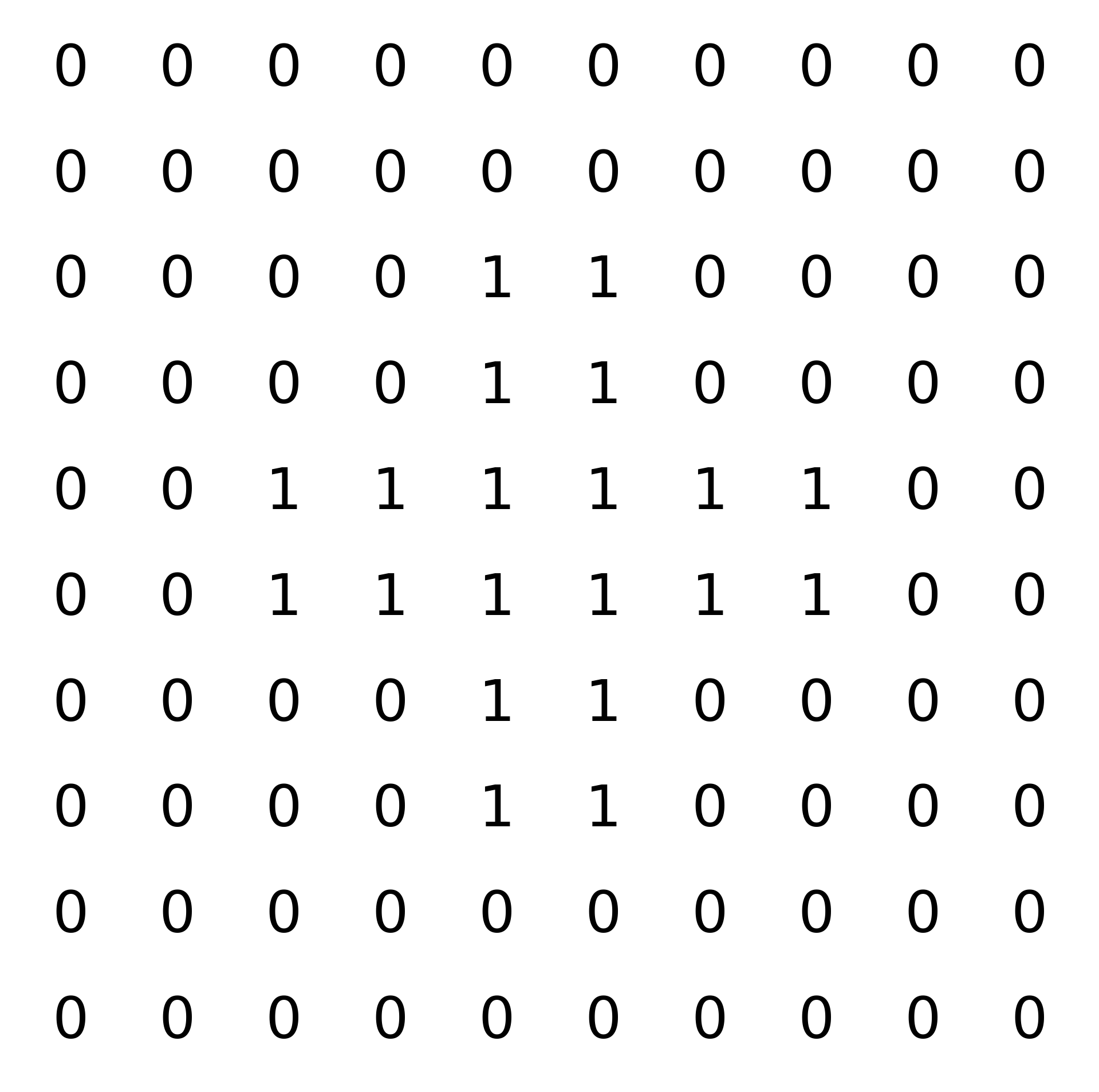}
\caption{CROSS}
\end{subfigure}
\caption{Ground truth images}
\label{fig:carma_matrix}
\end{figure}
For each image type, we consider two sizes: small-size $15 \times 15$ images and medium-size $100 \times 100$ images.
To generate blurred images, as in~\cite{crama2017class}, we employ the bit-flipping noise model defined in Section~\ref{introduction} with $p \in [0.1:0.1:0.5]$, where $[0.1:0.1:0.5]$ denotes a regularly-spaced vector between $0.1$ and $0.5$ using $0.1$ as the
 increment between elements. For each fixed $p$, we generate 50 random instances and report average relative optimality gaps and average CPU times. All LPs and IPs are solved with {\tt Gurobi~11.0}, where all options are set to their default values.
The relative optimality gaps and CPU times for different LP relaxations are shown in Figure~\ref{figure:denoising gap} and Figure~\ref{figure:denoising time}, respectively. \emph{In all instances, the clique LP returns a binary solution}; \ie Problem~\ref{cliqueLP} solves the original nonconvex Problem~\ref{map6}. Therefore, for this test set, there is no need to run the multi-clique LP. As can be seen from Figure~\ref{figure:denoising gap}, the standard LP performs poorly in all cases. The flower LP leads to a moderate improvement in the quality of bounds, while the running LP significantly outperforms the flower LP. However, Figure~\ref{figure:denoising time} indicates that the computational cost of solving the running LP is significantly higher than other LP relaxations. Hence, for this test set, the clique LP is the best relaxation as it always solves the original problem and has the lowest computational cost among competitors.

\begin{figure}[htbp]
\centering
\begin{subfigure}[t]{0.43\textwidth}
\centering
\includegraphics[width=\linewidth]{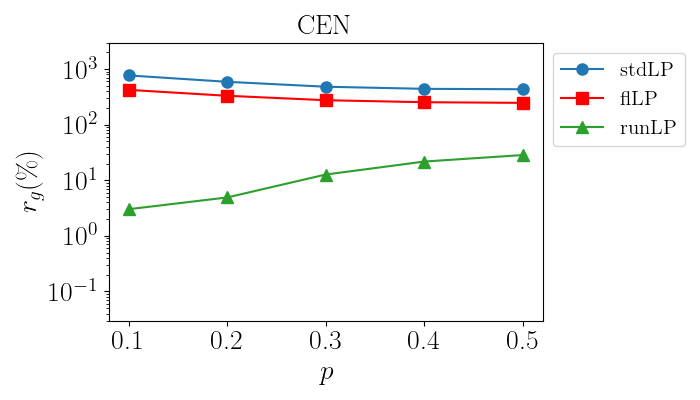}
\includegraphics[width=\linewidth]{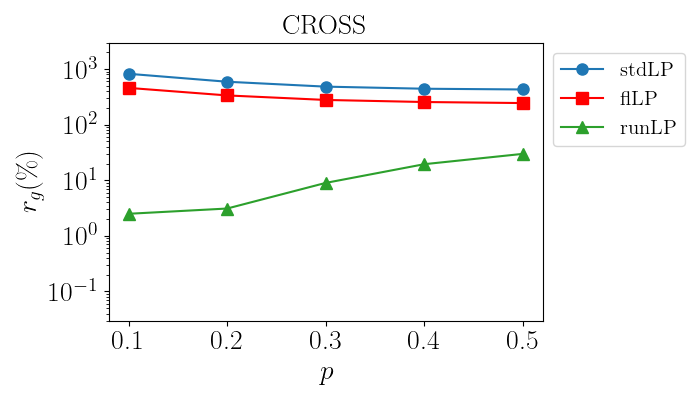}
\includegraphics[width=\linewidth]{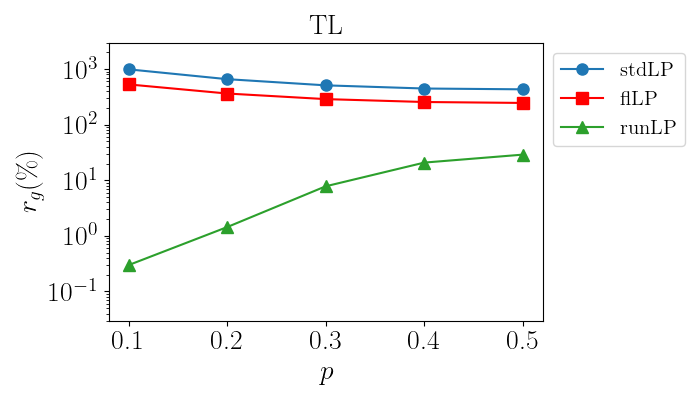}
\caption{$15\times 15$ images}
\end{subfigure}\hfill
\begin{subfigure}[t]{0.43\textwidth}
\centering
\includegraphics[width=\linewidth]{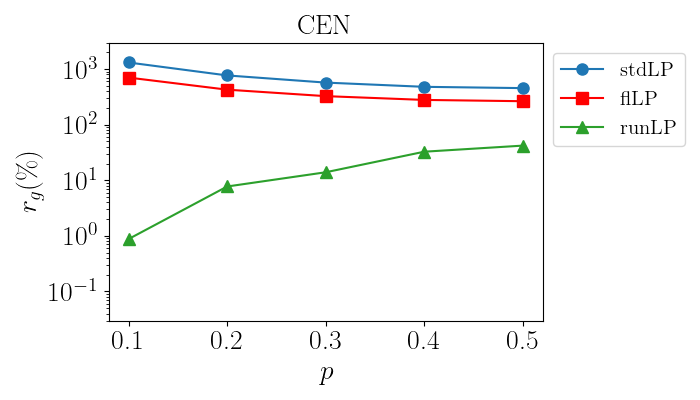}
\includegraphics[width=\linewidth]{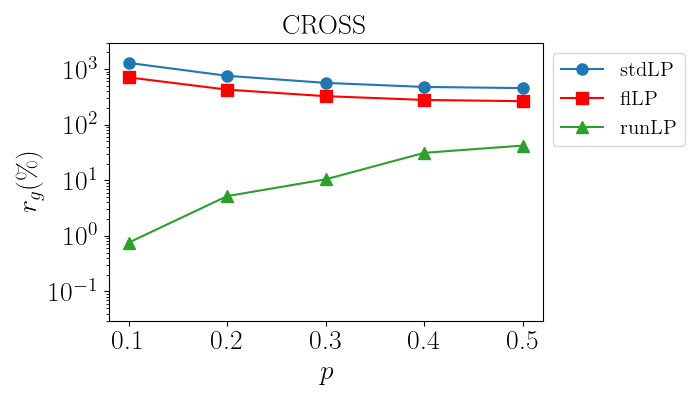}
\includegraphics[width=\linewidth]{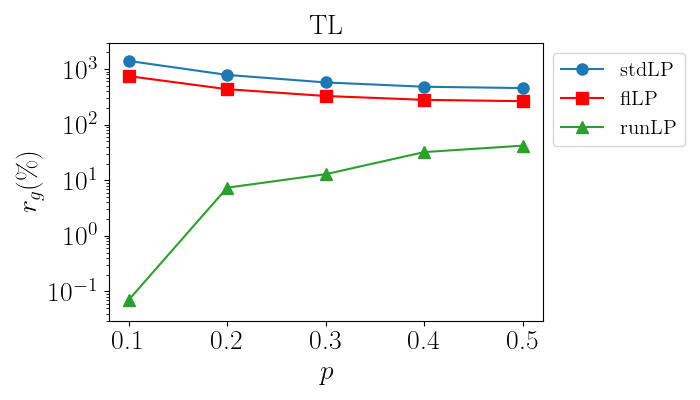}
\caption{$100\times 100$ images}
\end{subfigure}
\caption{Relative optimality gap of LPs for synthetic images. The clique LP is not depicted in these pictures as this LP returns a binary solution in all cases, implying a zero optimality gap.}
\label{figure:denoising gap}
\end{figure}

\begin{figure}[htbp]
\centering
\begin{subfigure}[t]{0.44\textwidth}
\centering
\includegraphics[width=\linewidth]{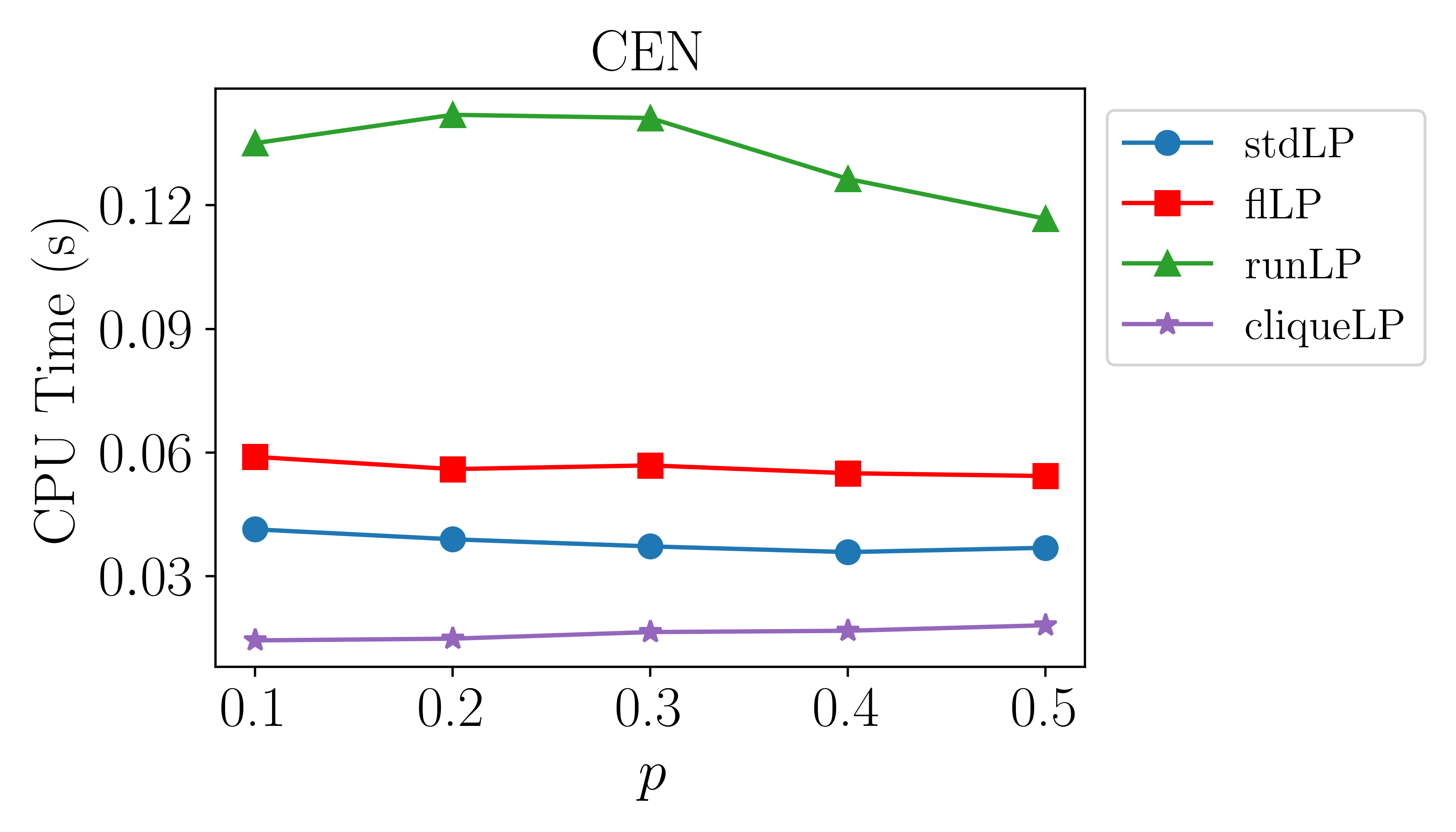}
\includegraphics[width=\linewidth]{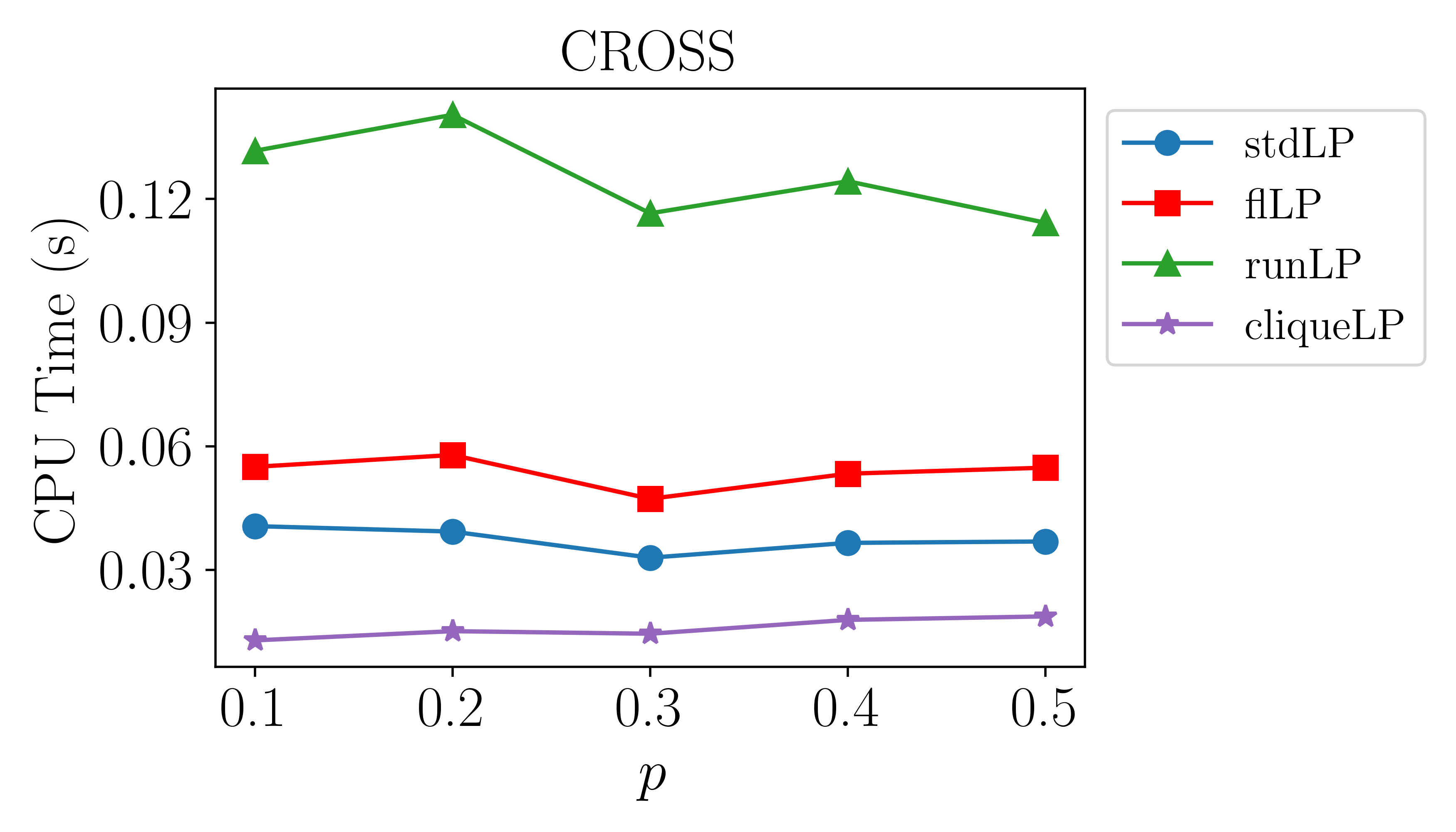}
\includegraphics[width=\linewidth]{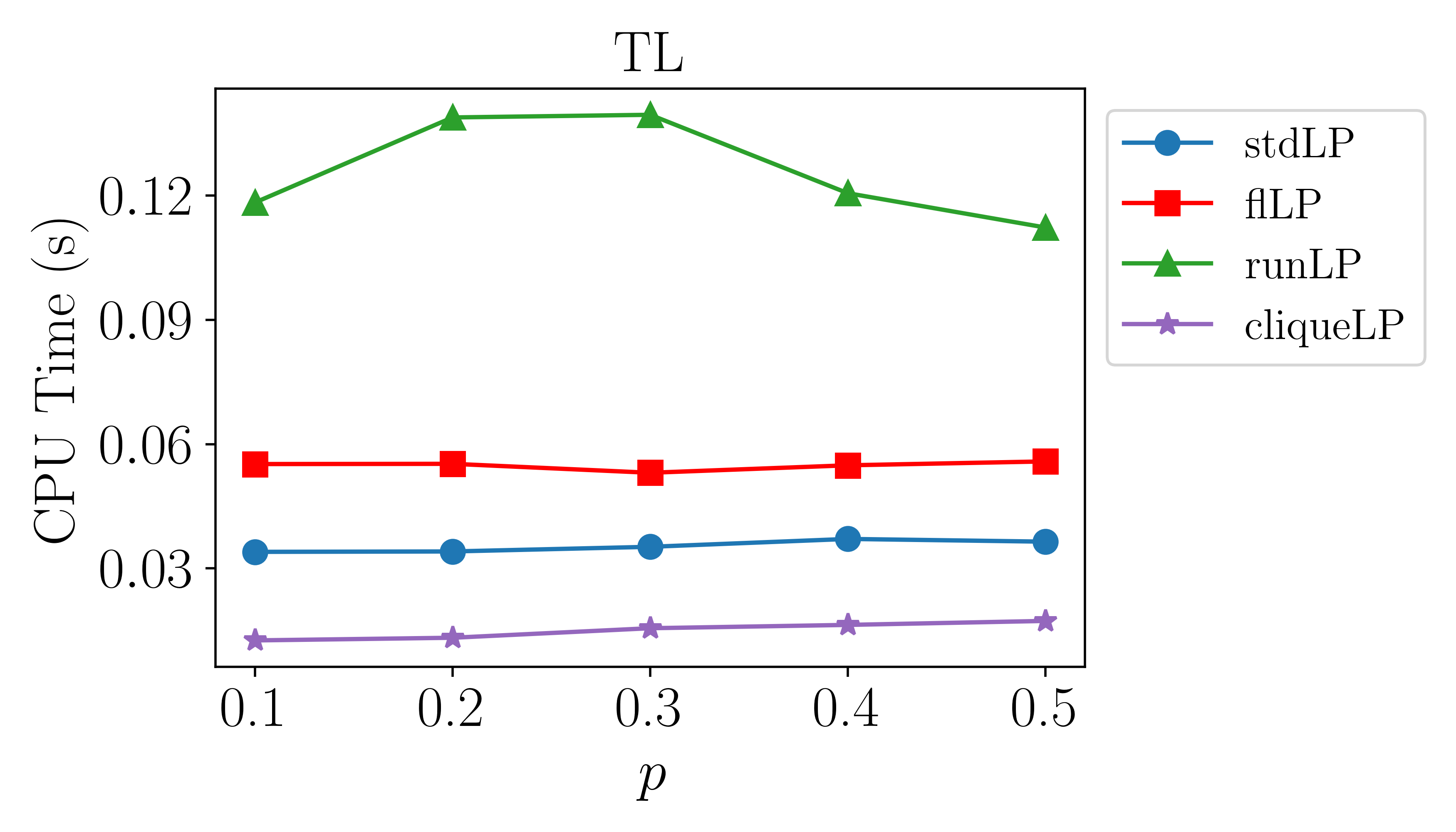}
\caption{$15\times 15$ images}
\end{subfigure}\hfill
\begin{subfigure}[t]{0.44\textwidth}
\centering
\includegraphics[width=\linewidth]{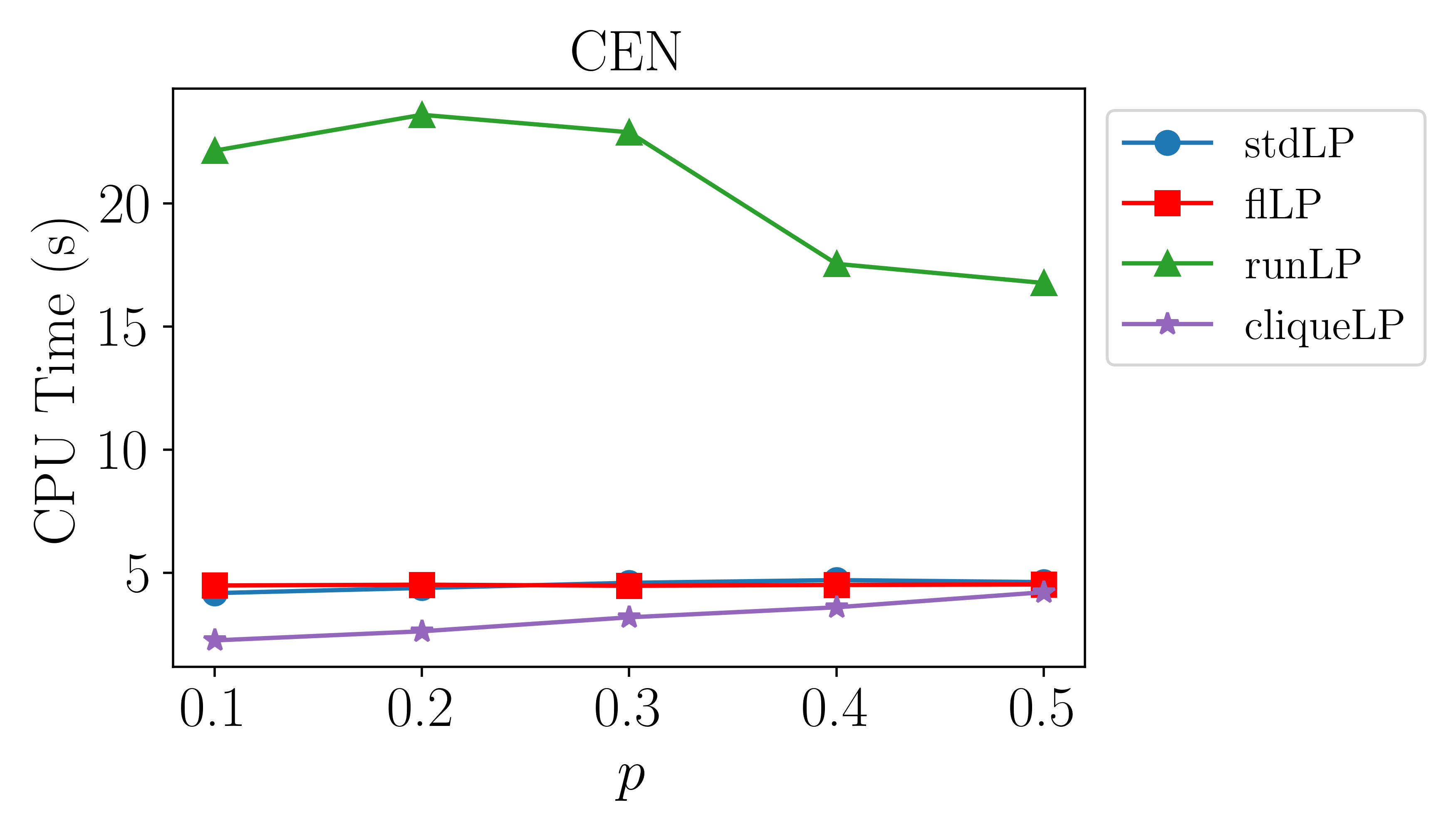}
\includegraphics[width=\linewidth]{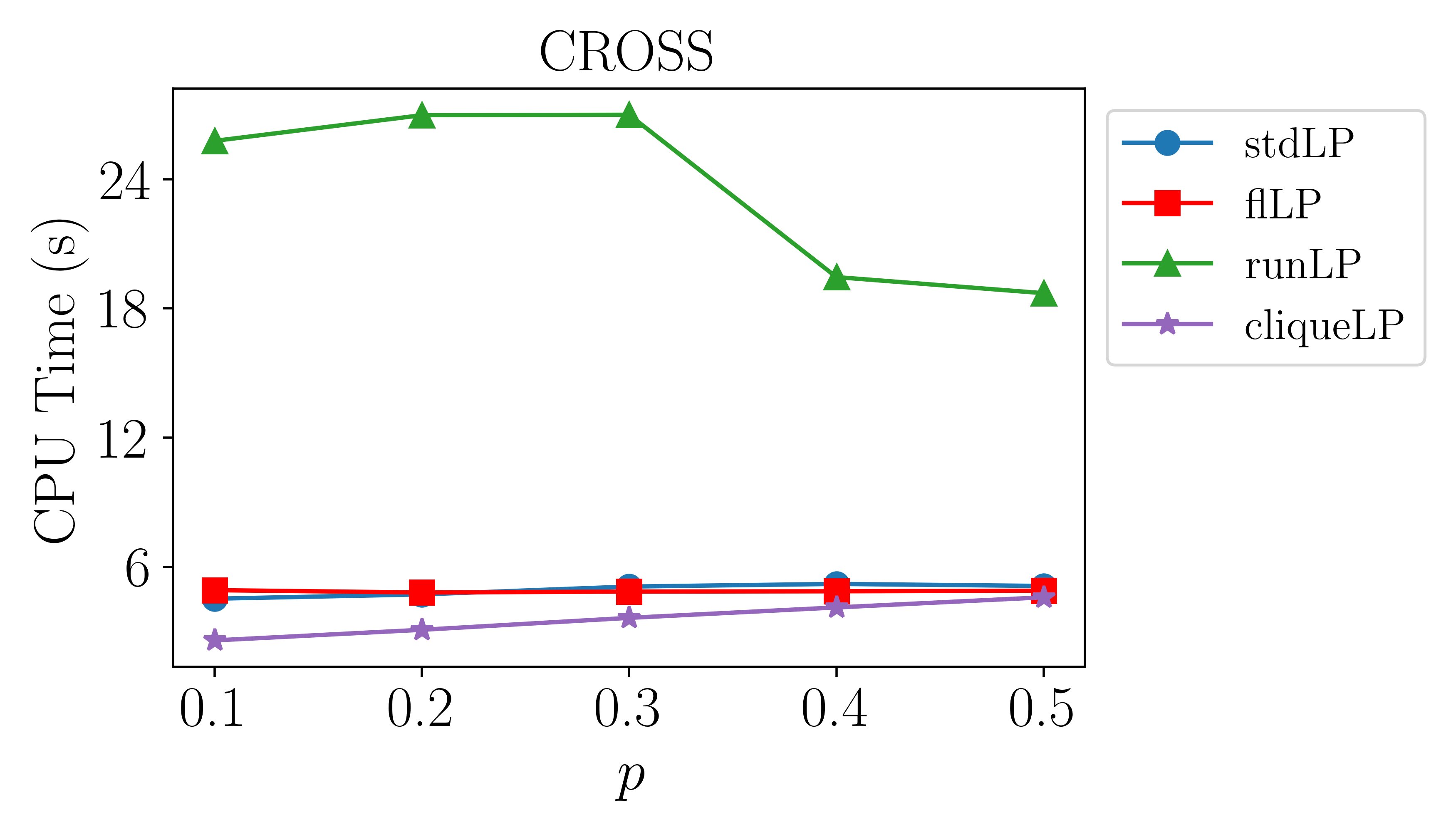}
\includegraphics[width=\linewidth]{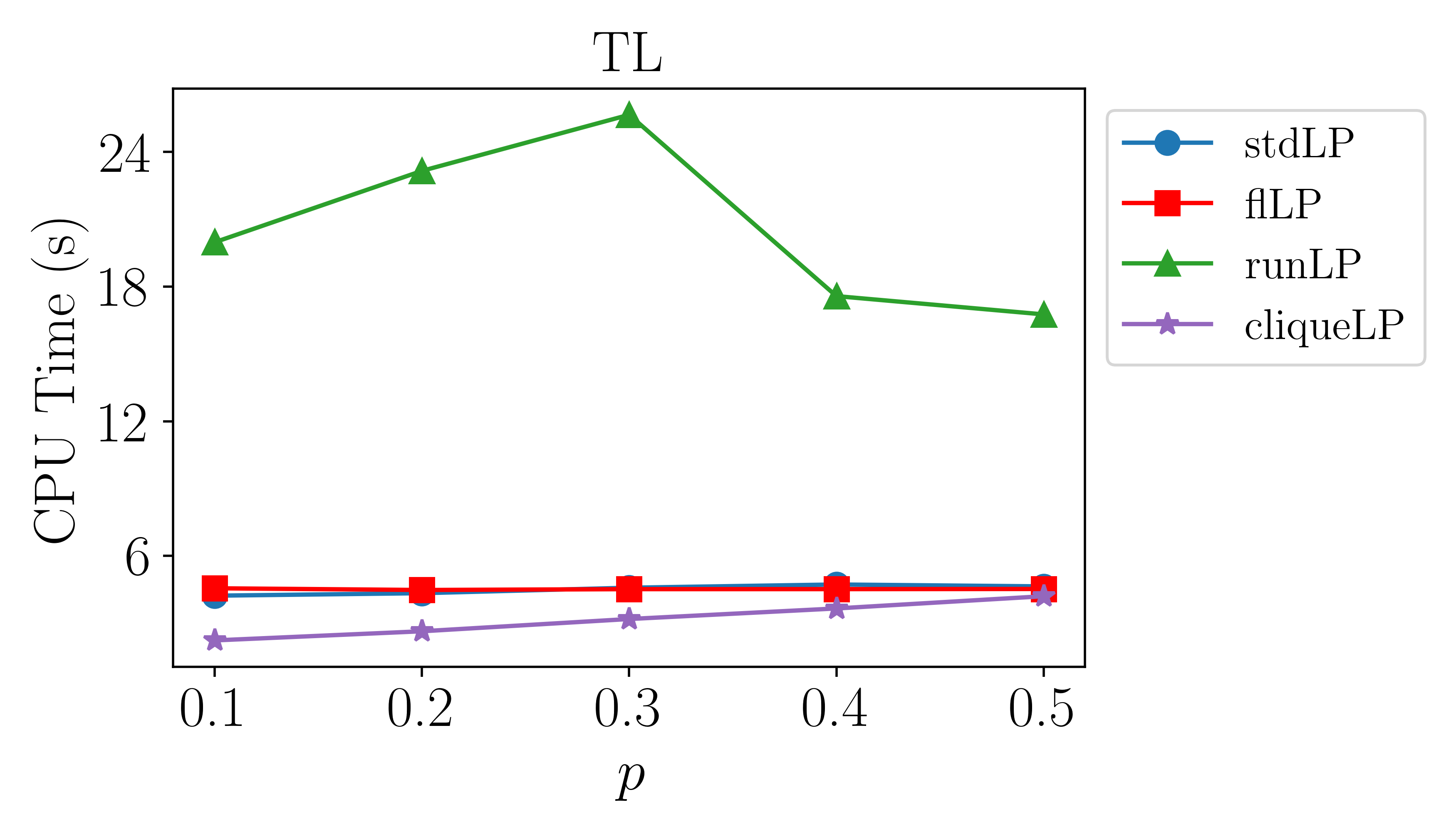}
\caption{$100\times 100$ images}
\end{subfigure}
\caption{CPU time of LP relaxations for synthetic images.}
\label{figure:denoising time}
\end{figure}

To illustrate the importance of constructing strong LP relaxations for Problem~\ref{map6}, let us comment on the computational cost of solving the IP; as we mentioned before, the IP is obtained by adding binary requirements for all variables to Problem~\ref{stdLP}. We consider a $100\times100$ ground truth image of type CEN and for each $p \in [0.1:0.1:0.5]$ we generate a blurred image. We then use {\tt Gurobi 11.0} to solve the IP and we set the time limit to 1800 seconds.  Results are summarized in Table~\ref{table1}; none of the IP instances are solved to optimality within the time limit; we also report the relative gap of the IP upon termination. As can be seen from the table, in all cases the relative gap is larger than $200\%$. Interestingly, in all cases the clique LP returns a binary solution in less than $5$ seconds.

\begin{table}[htpb]
\centering
\caption{Performance of {\tt Gurobi 11.0} when solving the integer program for $100\times 100$ images of type CEN. The integer program is obtained by adding the integrality constraints to the standard LP.}
\begin{tabular}{|c|c|c|c|}
\hline
$p$ & IP time (s)& gap $(\%)$ & clique LP time (s) \\ \hline
$0.1$ & $> 1800$ & $263$  & $2.04$\\
$0.2$ & $> 1800$ & $250$  & $2.36$\\
$0.3$ & $> 1800$ & $231$  & $4.17$\\
$0.4$ & $> 1800$ & $220$  & $4.25$\\
$0.5$ & $> 1800$ & $216$  & $4.42$\\ \hline
\end{tabular}
\label{table1}
\end{table}

\subsection{QR codes}
In this section, we demonstrate the effectiveness of the proposed LPs to restore an important type of real-world images: QR Codes. To this end, we utilize a more systematic approach for setting parameters of the inference problem so that we examine the quality of restored images. Henceforth, given a ground truth image and an algorithm for solving the image restoration problem, we measure the quality of the restored image in terms of \emph{partial recovery}; \ie the fraction of pixels that are identical in ground truth and restored images.
Whenever an LP returns a fractional solution, we first round the solution to the closest binary point and then compute partial recovery.

To capture the smoothness information for QR codes, we choose to learn the potential values from  small-size and ``mildly blurred'' images. That is, we generate $10$ distinct $50\times50$ QR codes and we use the bit-flipping noise model with $p=0.05$ to generate mildly blurred instances.
Next, we compute the average fraction of times $f_i$, $i \in [4]$ each group listed in Table~\ref{table:patterns} appears in these images. We then set $\varphi_i = - f_i$ for all $i \in [4]$.
While this method is unrealistic for real-world problems as it assumes we have the ground truth image at hand, it imitates practical approaches in which practitioners consider a large database of somewhat clean images to learn the frequency of different potential patterns. Using more sophisticated techniques to learn potential values is beyond the scope of this paper.

Next, we describe how to choose parameter $\alpha$; recall that $\alpha$ balances the similarity of the restored image to the blurred image and the smoothness of the restored image.
We choose $\alpha$ that maximizes the average partial recovery over a set of small-size $50\times 50$ QR codes. More precisely, we generate $10$ distinct QR codes; for each ground truth image, we set $p \in \{0.1, 0.2, 0.3\}$ and for each fixed $p$ we generate 50 random blurred images. We set $\alpha \in [0.1:0.1:1.5]$ and for each fixed $\alpha$ we solve the IP. Notice that since we are considering $50\times 50$ QR codes, {\tt Gurobi} is able to solve the IP in a few seconds. For each $(p, \alpha)$, we compute the average partial recovery over 50 instances and choose $\alpha$ that maximizes this quantity. Results are depicted in Figure~\ref{fig:qr_alpha}; accordingly, we set $\alpha = 1.0$ for our next tests.

\begin{figure}[htbp]
    \centering
    \begin{subfigure}[b]{0.4\textwidth}
        \includegraphics[width=\linewidth]{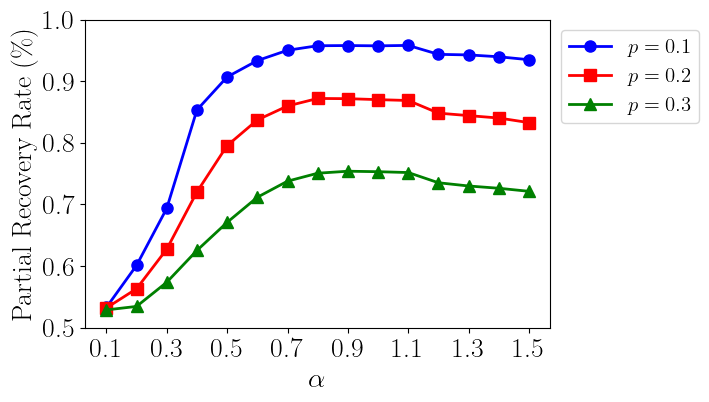}
    \end{subfigure}
    \caption{Learning parameter $\alpha$ using $50\times 50$ QR codes}
    \label{fig:qr_alpha}
\end{figure}

To test the performance of the proposed LP relaxations for restoring QR codes, we construct as the ground truth a $200\times 200$ QR code which contains a $100$-character-long text string. We set $p \in [0.1:0.1:0.5]$ and for each fixed $p$, we generate $50$ random instances. Results are shown in Figure~\ref{fig:qr_testing}. In addition to partial recovery rate, relative optimality gap, and CPU time, we also compare the \emph{tightness rate} of different LPs; we define the tightness rate as the fraction of times each LP returns a binary solution. As can be seen from these figures, for this test set, the standard LP and the flower LP perform quite poorly, whereas, the running LP, the clique LP and the multi-clique LP perform very well.
Namely, the partial recovery rate of these three LPs is very close to that of the IP.
Interestingly, the multi-clique LP has the best tightness rate; however, in many instances for which the running LP and the clique LP return fractional solutions, the relative optimality gaps are very small, and the rounded binary solutions lead to similar partial recovery values to those of the multi-clique LP. As before, the computational cost of solving the running LP is significantly higher than other LPs. Hence for this test set the multi-clique LP is the best option, followed closely by the clique LP.
\begin{figure}[htbp]
    \centering
    \begin{subfigure}[b]{0.4\textwidth}
        \includegraphics[width=\linewidth]{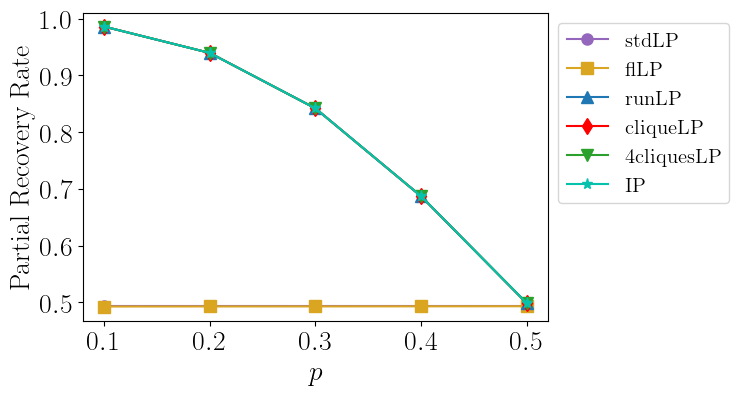}
        \caption{Partial recovery rate}
    \end{subfigure}
    \hspace{0mm}
    \begin{subfigure}[b]{0.4\textwidth}
        \includegraphics[width=\linewidth]{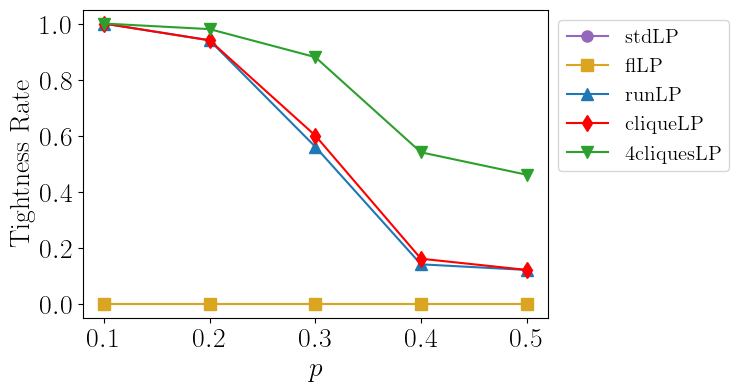}
        \caption{Tightness rate}
    \end{subfigure}

    \vspace{5mm}
    \hspace{0mm}
    \begin{subfigure}[b]{0.4\textwidth}
        \centering
        \includegraphics[width=\linewidth]{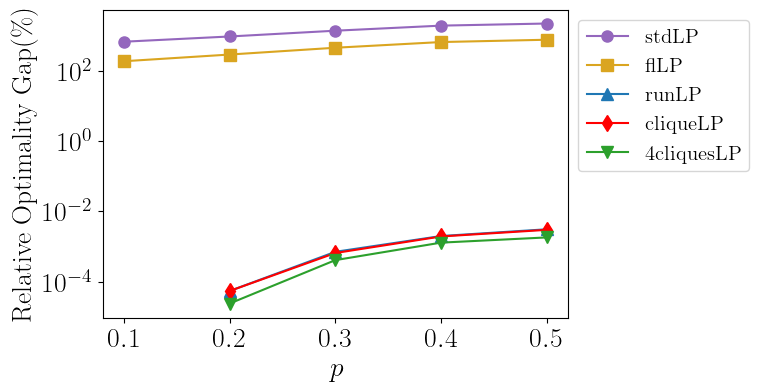}
        \caption{Relative optimality gap}
    \end{subfigure}
    \begin{subfigure}[b]{0.4\textwidth}
        \centering
        \includegraphics[width=\linewidth]{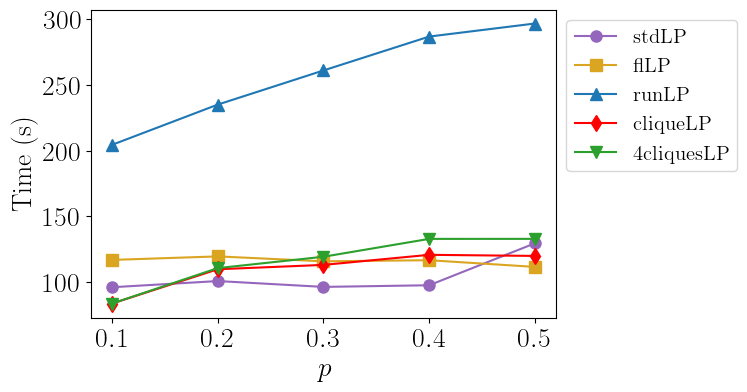}
        \caption{CPU time (s)}
    \end{subfigure}
    \caption{Performance of different LPs for restoring QR codes}
     \label{fig:qr_testing}
\end{figure}

Figure~\ref{fig:QRcodes} shows the $200\times200$ ground truth QR code, together with a noisy instance with $p=0.2$ and the restored QR code obtained by solving the clique LP.
While the noisy QR code (Figure~\ref{fig:QRcodesNoisy}) does not scan, the restored QR code (Figure~\ref{fig:QRcodesRestored}) scans successfully.

\begin{figure}[htbp]
\centering
\begin{subfigure}[b]{0.15\textwidth}
\centering
\includegraphics[width = \linewidth]{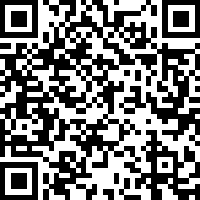}
\caption{original}
\end{subfigure}
\hspace{25mm}
\begin{subfigure}[b]{0.15\textwidth}
\centering
\includegraphics[width = \linewidth]{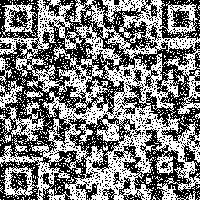}
\caption{noisy}
\label{fig:QRcodesNoisy}
\end{subfigure}
\hspace{25mm}
\begin{subfigure}[b]{0.15\textwidth}
\centering
\includegraphics[width = \linewidth]{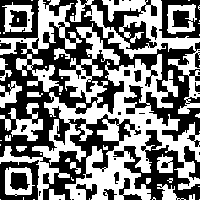}
\caption{restored}
\label{fig:QRcodesRestored}
\end{subfigure}
\caption{Restoring QR codes with LP relaxations.}
\label{fig:QRcodes}
\end{figure}

\section{Second application: decoding error-correcting codes}
\label{sec:decoding}

Transmitting a message, represented as a sequence of binary numbers, across a noisy channel is a central problem in information theory. The received message is often different from the original one due to the presence of noise in the channel and the goal in decoding is to recover the ground truth message. To this end, a common strategy is to transmit some redundant bits along with the original message, containing additional information about the message, so that some of the errors can be corrected. Such messages are often referred to as \emph{error-correcting} codes. Low Density Parity Check (LDPC) codes, first introduced by Gallager~\cite{gallager1962low}, are a popular type of error-correcting codes in which additional information is transmitted via \emph{parity bits}; it has been shown that LDPC codes enjoy various desirable theoretical and computational properties~\cite{richard2001,feldman2005using,feldman2006lp}. Existing methods for decoding LDPC codes are based on the belief propagation algorithm~\cite{mceliece98} and LP relaxations~\cite{feldman2005using,feldman2006lp}.

LDPC codes are often represented via UGMs; namely, each node of the graph corresponds to a message bit while each clique corresponds to a subset of bits with even parity. Gallager~\cite{gallager1962low} introduced LDPC codes as error-correcting codes with three properties: $(i)$ all cliques have the same cardinality, denoted by $\beta$, $(ii)$ each node appears in the same number of cliques, denoted by $\gamma$, and
$(iii)$ $\beta > \gamma$. Denoting by $n$ the number of message bits, an LDPC code is fully characterized by the triplet $(n, \beta, \gamma)$. Figure~\ref{fig:LDPC}  illustrates an LDPC code with $(n, \beta, \gamma) = (9,3,2)$.

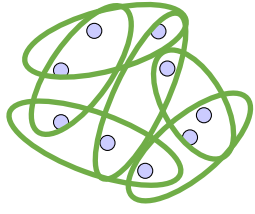
\begin{figure}
\begin{center}
    \begin{tikzpicture}[scale=0.5]
    \tikzset{every node/.style={circle, draw=black, fill=blue!20, inner sep=2pt}}
    \node[] (8) at ({2*cos(0)-0.1}, {2*sin(0)-0.5}) {};
    \node[] (9) at ({2*cos(40)-0.6}, {2*sin(40)-0.55}) {};
    \node[] (2) at ({2*cos(80)+0.35}, {2*sin(80)-0.25}) {};
    \node[] (3) at ({2*cos(120)}, {2*sin(120)}) {};
    \node[] (1) at ({2*cos(160)}, {2*sin(160)}) {};
    \node[] (5) at ({2*cos(200)}, {2*sin(200)}) {};
    \node[] (6) at ({2*cos(240)+0.35}, {2*sin(240)+0.5}) {};
    \node[] (4) at ({2*cos(280)}, {2*sin(280)}) {};
    \node[] (7) at ({2*cos(320)}, {2*sin(320)+0.2}) {};

    \draw[line width=2pt, draw=lightgreen, rotate = -35] (-1.5,-0.2) ellipse (0.8cm and 2cm);
    \draw[line width=2pt, draw=lightgreen, rotate = -75] (-1.6,-0.3) ellipse (0.8cm and 2.2cm);
    \draw[line width=2pt, draw=lightgreen, rotate = -50] (2.0,0.3) ellipse (0.8cm and 2cm);
    \draw[line width=2pt, draw=lightgreen, rotate = -150] (-1.3,1.0) ellipse (0.8cm and 1.6cm);
    \draw[line width=2pt, draw=lightgreen, rotate = -25] (0.2,0.1) ellipse (0.7cm and 2.4cm);
    \draw[line width=2pt, draw=lightgreen, rotate = -110] (1.5,-0.5) ellipse (0.8cm and 2.4cm);
\end{tikzpicture}
\end{center}
\caption{The clique structure of a $(9, 3,2)$ LDPC code.}
\label{fig:LDPC}
\end{figure}

Now let us formalize the problem of decoding LDPC codes. Consider a ground truth message $x_v$, $v \in V$. Denote by $\oplus$ the addition in modulo two arithmetic.
Then for each $C \in \C$ we must have $\oplus_{v \in C}{z_v} = 1$. That is, we define the clique potentials $\phi_C(z_C)$, $C \in \C$ as follows:
\begin{equation*}
    \phi_C(z_C) = \left\{
    \begin{array}{ll}
         1   & {\rm if} \; \oplus_{v \in C}{z_v} = 1\\
         0  & {\rm otherwise.}
    \end{array}
    \right.
\end{equation*}
The clique potentials can be equivalently written as:
\begin{equation*}
    \phi_C(z_C) = \sum_{\substack{S \subseteq C \\ |S| \; {\rm even}}}{\Big(\prod_{v \in S}{z_v} \prod_{v \in C \setminus S}{(1-z_v)}\Big)}.
\end{equation*}
Denoting the noisy message by $y_v$, $v \in V$ and assuming the bit-flipping model for the noisy channel, we deduce that the decoding problem for an LDPC code can be written as:
\begin{align}\label{map7}\tag{DCD}
    {\rm max} \quad & \sum_{\substack{v\in V: \\ y_v=1}}{z_v}-\sum_{\substack{v\in V:\\ y_v=0}}{z_v}\\
    \st &\sum_{\substack{S \subseteq C \\ |S| \; {\rm even}}}{\Big(\prod_{v \in S}{z_v} \prod_{v \in C \setminus S}{(1-z_v)}\Big)} = 1, \quad \forall C \in \C\nonumber\\
    & z_v \in \{0,1\}, \forall v \in V. \nonumber
\end{align}
Denote by $\D(\C)$ the feasible region of Problem~\ref{map7}.
As before we introduce auxiliary variables $z_e := \prod_{v \in e}{z_v}$ for all $e \in \bar P(C)$ and $C \in \C$ to obtain the following reformulation of Problem~\ref{map7} in an extended space:
\begin{align}\label{map8}\tag{L-DCD}
    {\rm max} \quad & \sum_{\substack{v\in V: \\ y_v=1}}{z_v}-\sum_{\substack{v\in V:\\ y_v=0}}{z_v}\\
    \st \quad  &\sum_{\substack{p \subseteq C:\\ p\neq \emptyset}}{(-2)^{|p|-1} z_p} = 0, \quad \forall C \in \C\nonumber\\
    & z \in S(\G), \nonumber
\end{align}
where $S(\G)$ is the multilinear set of the UGM hypergraph $\G=(V,\E)$, $\E = \cup_{C \in \C}{\bar P(C)}$ and is defined by~\eqref{multSet}. By replacing the nonconvex set $S(\G)$ in $z \in S(\G)$ with the polyhedral relaxations introduced in Section~\ref{sec:lprelaxations}, we obtain various LP relaxations for Problem~\ref{map8}. In the following, we denote by $S^e(\G)$ the feasible region of Problem~\ref{map8} and by $\MP^e(\G)$ its convex hull.

\subsection{The clique relaxation for decoding}

The clique LP for Problem~\ref{map8} is obtained by replacing the constraint $z \in S(\G)$ with $z \in \MP^{\rm cl}(\G)$, where $\MP^{\rm cl}(\G)$ is the clique relaxation and is defined in Section~\ref{sec:cliqueLP}. In the following, we show that the clique LP for Problem~\ref{map8} has an interesting interpretation; namely, it is obtained
by replacing the feasible region of Problem~\ref{map8} corresponding to a single clique $C$ with its convex hull.

\begin{prop}\label{cnvr}
    Let $\G_C$ denote the complete hypergraph with node set $C$. Consider the set
    $$
    S^{e}(\G_C) = \Big\{z \in S(\G_C): \sum_{\substack{p \subseteq C:\\ p\neq \emptyset}}{(-2)^{|p|-1} z_p} = 0\Big\}.
    $$
    Then the convex hull of $S^{e}(\G_C)$  is given by:
    \begin{equation}\label{convNew}
    \MP^{e}(\G_C) = \Big\{z \in \MP(\G_C): \sum_{\substack{p \subseteq C:\\ p\neq \emptyset}}{(-2)^{|p|-1} z_p} = 0\Big\}.
    \end{equation}
\end{prop}

\begin{proof}
To prove the statement, it suffices to show that $\MP^{e}(\G_C)$ is an extended formulation for the convex hull of the set:
    $$
    \D(C) := \Big\{z \in \{0,1\}^C: \sum_{\substack{S \subseteq C \\ |S| \; {\rm even}}}{\Big(\prod_{v \in S}{z_v} \prod_{v \in C \setminus S}{(1-z_v)}\Big)} = 1\Big\}.
    $$
To construct the convex hull of $\D(C)$, we make use of RLT as defined in~\cite{SheAda90}. That is, let $J_1, J_2$ be
any partition of $C$; define the factor $F(J_1,J_2) = \prod_{v \in J_1}{z_v} \prod_{v \in J_2}{(1-z_v)}$.
We first expand each $F(J_1, J_2) \geq 0$ and let
 $z_e = \prod_{v \in e}{z_v}$ for each product term to obtain linear inequalities~\eqref{eq:rlt}.
Subsequently, we multiply the equality constraint
\begin{equation}\label{longeq}
\sum_{\substack{S \subseteq C \\ |S| \; {\rm even}}}{\Big(\prod_{v \in S}{z_v} \prod_{v \in C \setminus S}{(1-z_v)}\Big)} = 1,
\end{equation}
by each factor $F(J_1,J_2)$ and let $z_e = \prod_{v \in e}{z_v}$ to obtain a collection of equations. Let us examine these equations; two cases arise:
\begin{itemize}[leftmargin=*]
    \item if $|J_1|$ is even, then multiplying $F(J_1,J_2)$ by equation~\eqref{longeq} and using $z_v (1-z_v) = 0$ for any $v \in C$, we obtain the trivial equation $F(J_1,J_2) = F(J_1,J_2)$.
    \item if $|J_1|$ is odd, then multiplying $F(J_1,J_2)$ by equation~\eqref{longeq} and using $z_v (1-z_v) = 0$ for any $v \in C$, we obtain $F(J_1, J_2) = 0$.
\end{itemize}
Therefore, by Section~4 of~\cite{SheAda90}, the following system defines an extended formulation for the convex hull of $\D(C)$:
$$
\psi_U(z_C) \geq 0, \;\forall U \subseteq C: |U|\;{\rm even},
\; \psi_U(z_C) = 0, \;\forall U \subseteq C: |U|\;{\rm odd},
$$
where $\psi_U(z_C)$ is defined by~\eqref{defRLT}. To complete the proof, it suffices to show that $\psi_U(z_C) \geq 0$ for all $U \subseteq C$ together with
\begin{equation}\label{longeq2}
\sum_{p \subseteq C, p\neq \emptyset}{(-2)^{|p|-1} z_p} = 0,
\end{equation}
implies $\psi_U(z_C) = 0$ for all $U \subseteq C$ such that $|U|$ is odd. To see this, first note that~\eqref{longeq2} can be equivalently written as
$\sum_{U \subseteq C: |U| {\rm even}}{\psi_U(z_C)} = 1$. Moreover, from the definition of $\psi_U(z_C)$ we have
$\sum_{U \subseteq C}{\psi_U(z_C)} = 1$. These two inequalities imply $\sum_{U \subseteq C: |U| {\rm odd}}{\psi_U(z_C)} = 0$, which together with $\psi_U(z_C) \geq 0$ for all $U \subseteq C$ yield $\psi_U(z_C) = 0$ for all
$U \subseteq C$ such that $|U|$ is odd.
\end{proof}

\subsection{The parity polytope and the parity LP}

In~\cite{jeroslow75}, Jeroslow proved that the convex hull of the set of binary vectors $x \in \{0,1\}^n$ with even parity, denoted by $\P_n$, is given by:
\begin{equation}
     \P_n = \Big\{x \in [0,1]^n:
     \sum_{i\in S} {(1-x_i)}+\sum_{i\in [n]\setminus S} {x_i}\geq 1, \; \forall S\subseteq [n] : |S| \; \text{is odd}\Big\}.
\end{equation}
Using this characterization, the authors of~\cite{feldman2005using}, introduced the following LP relaxation of Problem~\ref{map7}, which we will refer to as the \emph{parity LP}:
\begin{align}\label{parityLP}\tag{parLP}
    {\rm max} \quad & \sum_{\substack{v\in V: \\ y_v=1}}{z_v}-\sum_{\substack{v\in V:\\ y_v=0}}{z_v}\\
    \st\quad &\sum_{v\in S} {(1-z_v)}+\sum_{v\in C\setminus S} {z_v}\geq 1, \quad \forall S\subseteq C : |S| \; \text{is odd}, \; \forall C \in \C \nonumber\\
    & z_v \in [0,1], \forall v \in V. \nonumber
\end{align}
In the following we show that the clique LP is stronger than the parity LP, in general.
Denote by $\P(\C)$ the feasible region of Problem~\ref{parityLP}. If $\C$ consists of a single clique $C$, then by Proposition~\ref{cnvr}:
$$\P(C) = \proj_{z_v, v\in C}(\MP^e(\G_C)) = \proj_{z_v, v\in C}(\MP(\G_C) \cap \H_C),$$
where $\H_C$ denotes the set of points satisfying equality~\eqref{longeq2}.
This implies that the feasible region of the clique LP is contained in the feasible region of the parity LP. As we detail next, this containment is often strict. We first examine the strength of the clique LP. In~\cite{dPKha25}, the authors generalize the decomposition result of Theorem~\ref{decomposability} to account for additional constraints on multilinear sets:

\begin{theorem}[\cite{dPKha25}]
\label{cor constrained decomposition}
Let $\G$ be a hypergraph, and let $\G_1$,$\G_2$ be 
section hypergraphs of $\G$ 
such that $\G_1 \cup \G_2 = \G$ and $\G_1 \cap \G_2$ is a complete hypergraph.
Let $C(\G)$ be the set of points in $S(\G)$ that satisfy a number of constraints, each one containing only variables corresponding to the nodes and edges only in $\G_1$ or only in $\G_2$.
For $i=1,2$, let $C(\G_i)$ be the projection of $C(\G)$ in the space of $S(\G_i)$.
Then, $C(\G)$ is decomposable into $C(\G_1)$ and $C(\G_2)$.
\end{theorem}

Thanks to Theorem~\ref{cor constrained decomposition} and Proposition~\ref{cnvr}, we can employ a similar line of arguments as in the proof of Proposition~\ref{prop:charac} to obtain a sufficient condition for sharpness of the clique LP for decoding:

\begin{prop}\label{prop:characE}
     Denote by $\C$ the set of maximal cliques of a binary UGM. Let $\G = \cup_{C \in \C}{\G_C}$,  where $\G_C$ is a complete hypergraph with node set $C$. If $\C$ has the running intersection property, then
     $$\MP^e(\G) = \bigcap_{C \in \C}(\MP(\G_C) \cap \H_C).$$
\end{prop}
If $\C=\{C_1, C_2\}$, then $\C$ clearly has the running intersection property, implying that the clique relaxation coincides with the convex hull. However, as we show next, even in this simple case, the parity relaxation $\P(\C)$ does not coincide with the convex hull.
\begin{example}
    Let $\C=\{C_1, C_2\}$  with $C_1=\{v_1, v_2, v_3, v_4\}$
    and $C_2 = \{v_3,v_4,v_5,v_6\}$. It can be checked that the point $\tilde z_{v_1} = \tilde z_{v_2}= \tilde z_{v_5}= 0$,
    $\tilde z_{v_3}= \tilde z_{v_4}=\frac{1}{2}$ and $\tilde z_{v_6} = 1$ is feasible for  $\P(\C)$. Since the points in $C_1$ and $C_2$ should have even parity, we conclude that the points in $(C_1 \setminus C_2) \cup (C_2 \setminus C_1)$ should have even parity as well; that is, inequality $(1-z_{v_6})+z_{v_1}+z_{v_2}+z_{v_5} \geq 1$ is valid for $\D(\C)$. Substituting $\tilde z$ in this inequality yields $1-1+0+0+0 \not\geq 1$, implying that $\D(\C)$ is strictly contained in $\P(\C)$.
\end{example}

\subsection{Numerical Experiments}
In this section, we compare the performance of different LP relaxations for decoding LDPC codes. We first describe how LDPC codes are generated~\cite{gallager1962low}: an $(n, \beta, \gamma)$ LDPC code is often characterized by a \emph{parity-check matrix}, an $m \times n$ binary matrix with $m:= \frac{n}{\gamma}\times \beta$, where each row contains $\beta$ ones and each column contains $\gamma$ ones. To construct a parity check matrix, we start by creating a matrix with all ones arranged in descending order; the $i$th row contains ones in columns $(i-1)\beta + 1$ to $i \beta$. We then permute the columns of this matrix randomly and append it to the initial matrix. This permutation and appending is repeated $\gamma-1$ times to ensure each column contains $\gamma$ ones. The ones in each row of the matrix then correspond to the nodes of a clique in the UGM. It then follows that the UGM consists of $\frac{n}{\gamma}\times \beta$ cliques, each consisting of $\beta$ nodes.
We assume an all-zero code as the ground truth code.
As the first set of experiments, we consider a $(60,4,3)$ LDPC code. We use the bit-flipping noise with $p \in [0:0.01:0.2]$
and for each $p$ we generate $400$ random trials. We then compare the performance of different LPs with respect to tightness rate and partial recovery rate as defined before. For the multi-clique LP we set $m=4$; \ie in our tests lifted odd-cycle inequalities~\eqref{liftedoddCycle} for cycles of cliques of length three and four are generated.
Results are shown in Figure~\ref{fig:43LDPC}. As can be seen from this figure, the standard LP and the flower LP perform quite poorly, while the running LP, the clique LP, the multi-clique LP, and the parity LP perform well. The multi-clique LP is the best, followed by the clique LP, followed by the running LP, followed by the parity LP. As before the computational cost of solving the running LP is significantly higher than other LPs. Motivated by these observations, in the next set of experiments we restrict our study to the parity LP, the clique LP, and the multi-clique LP.

\begin{figure}[htbp]
\centering
\begin{subfigure}[b]{0.48\textwidth}
\centering
    \includegraphics[width=\linewidth]{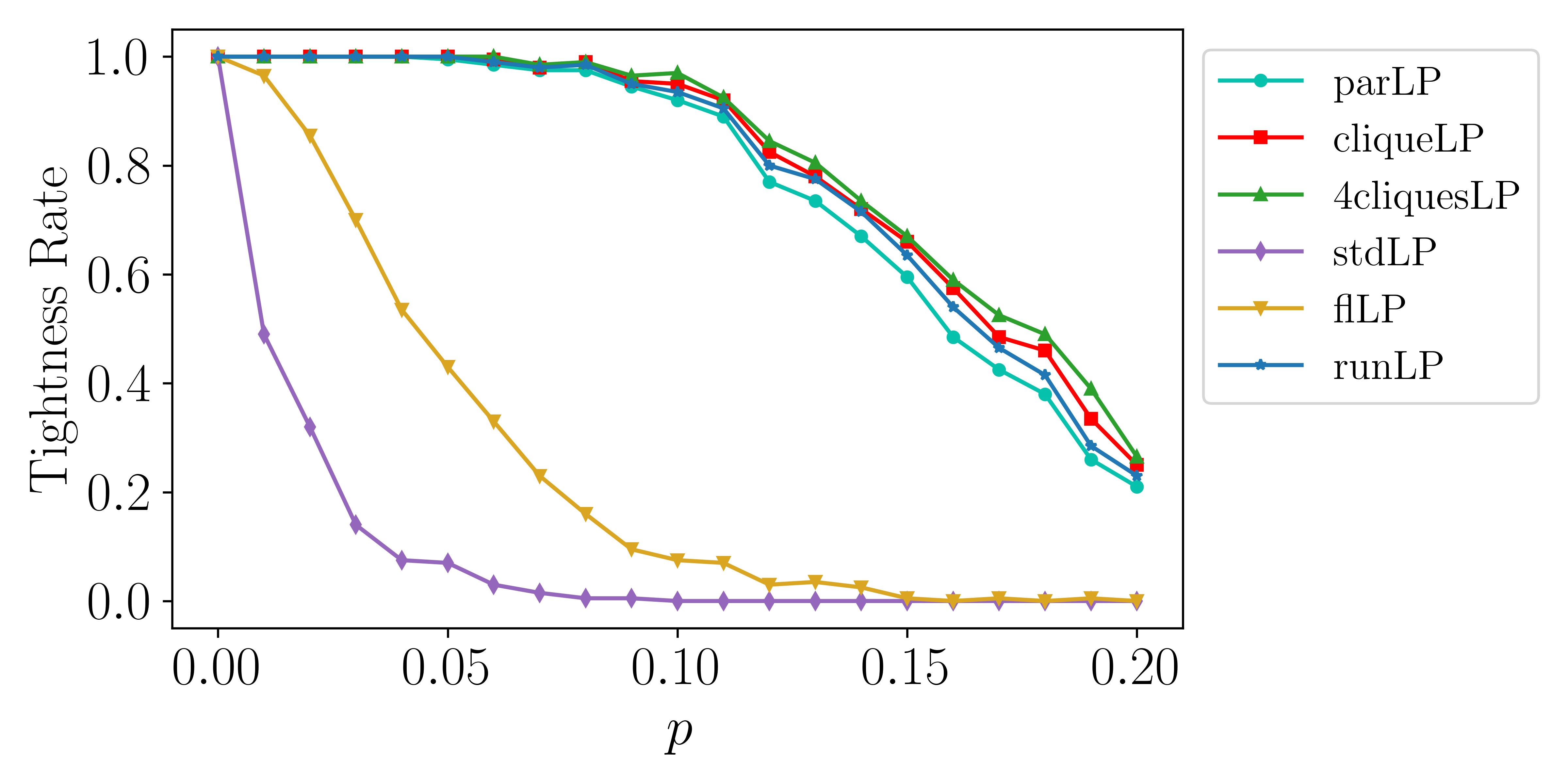}
\end{subfigure}
\hspace{2mm}
\begin{subfigure}[b]{0.48\textwidth}
\centering
\includegraphics[width = \linewidth]{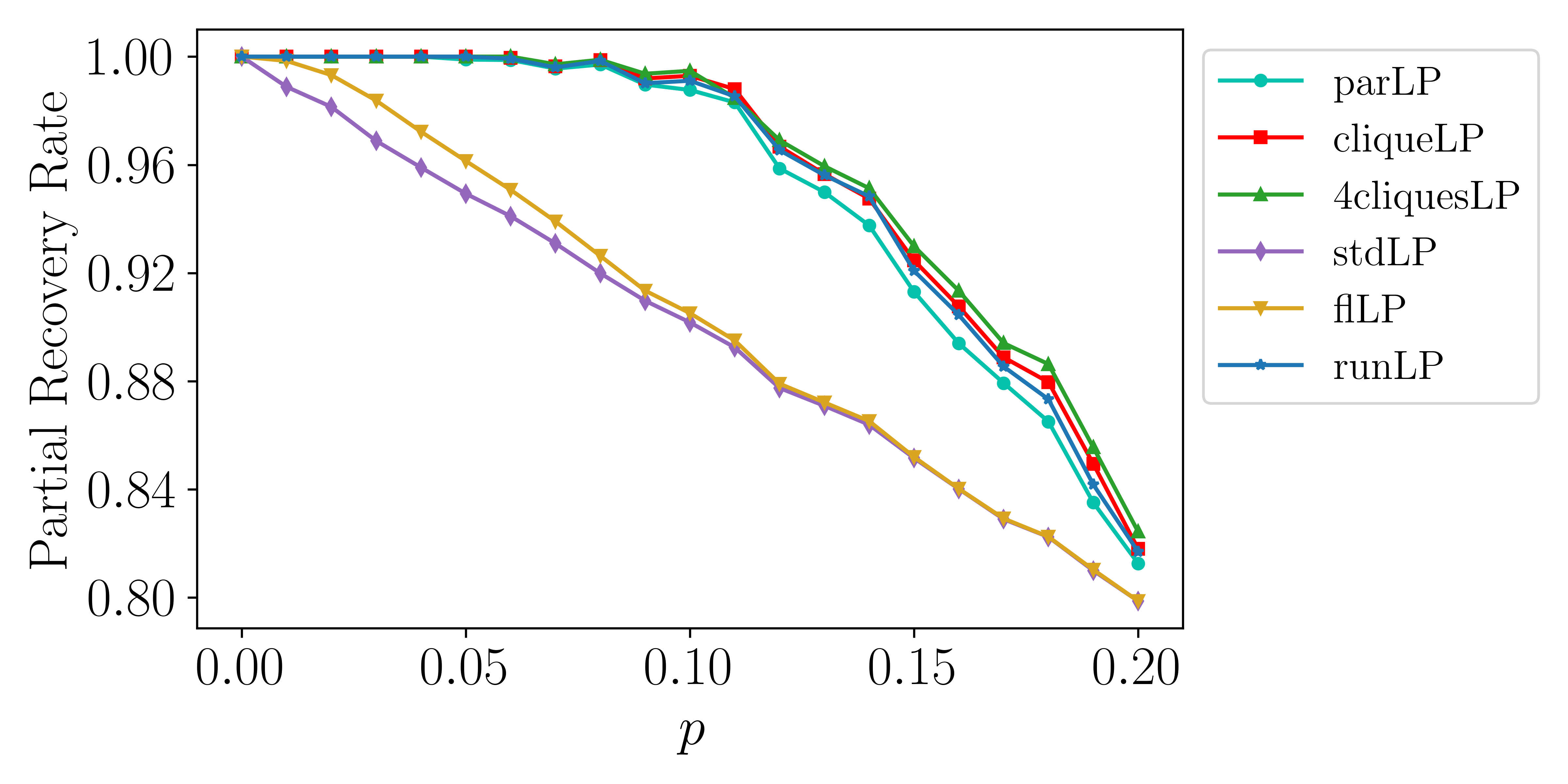}
\end{subfigure}
\caption{Performance of different LPs for decoding $(60,4,3)$ LDPC codes.}
\label{fig:43LDPC}
\end{figure}

We next consider three type of LDPC codes: $(120,4,3)$,
$(120,5,4)$, and $(120,6,5)$. We set $p \in [0:0.01:0.2]$ and for each $p$ we generate $400$ random trials. Results are depicted in Figure~\ref{fig:longLDPCs}. Overall, the multi-clique LP and the clique LP have better tightness rates than the parity LP, as the theory suggests. However, the differences, especially in terms of partial recovery rates  become smaller as we increase the code length. This indeed, indicates the difficulty of solving this problem class; namely, by constructing stronger LP relaxations, the partial recovery rate of the decoder only marginally improves. These results also suggest that for a fixed code length, as we increase the clique size, the performance of all LPs degrade; while for a $(120,4,3)$ LDPC code all LPs manage to recover the ground truth with up to about $10\%$ corruption, for a $(120,6,5)$ LDPC code, this number decreases to about $5\%$ corruption.

\begin{figure}[htbp]
\centering
\begin{subfigure}[t]{\textwidth}
\centering
\includegraphics[width=0.48\linewidth]{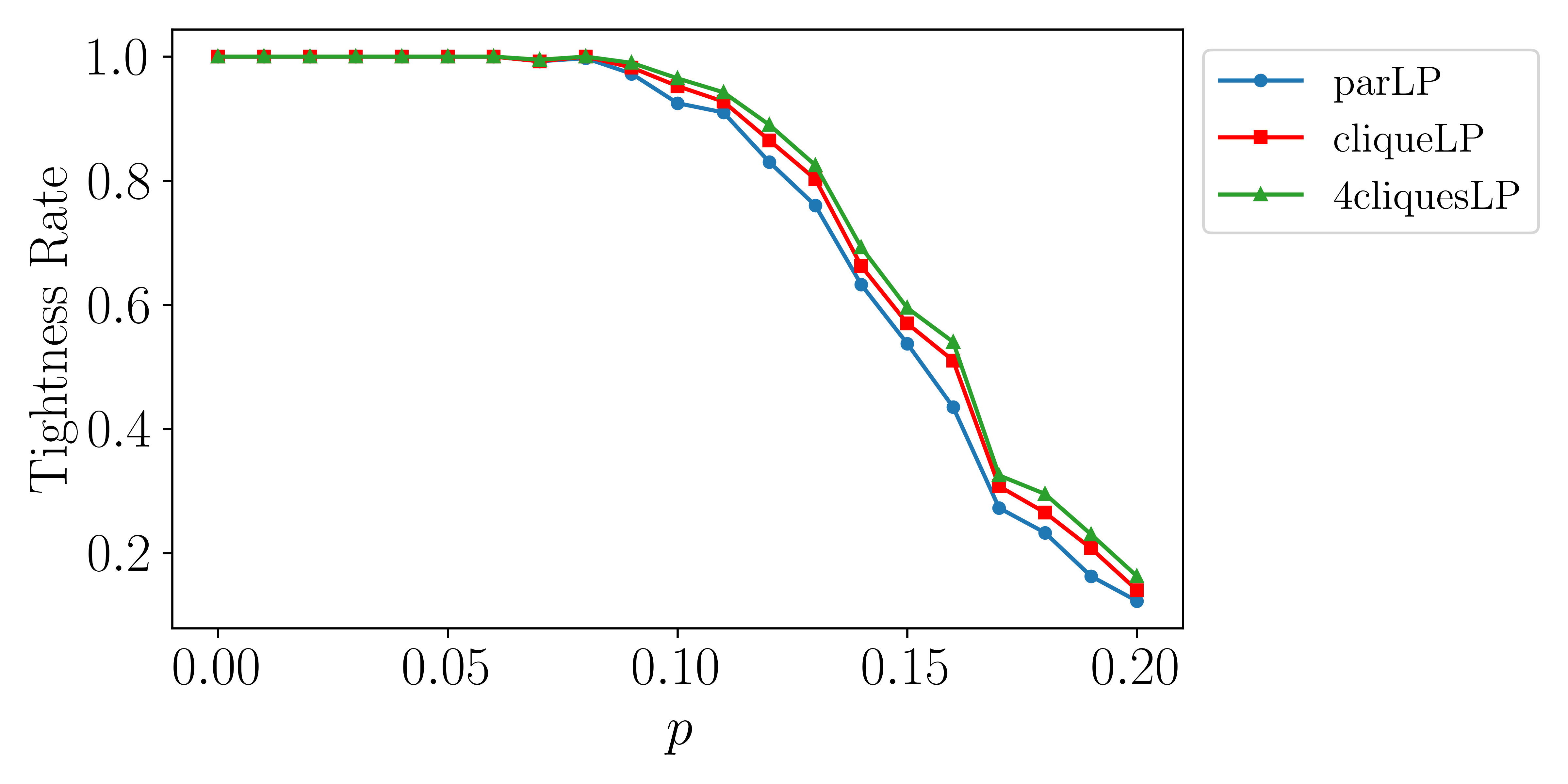}
\includegraphics[width=0.48\linewidth]{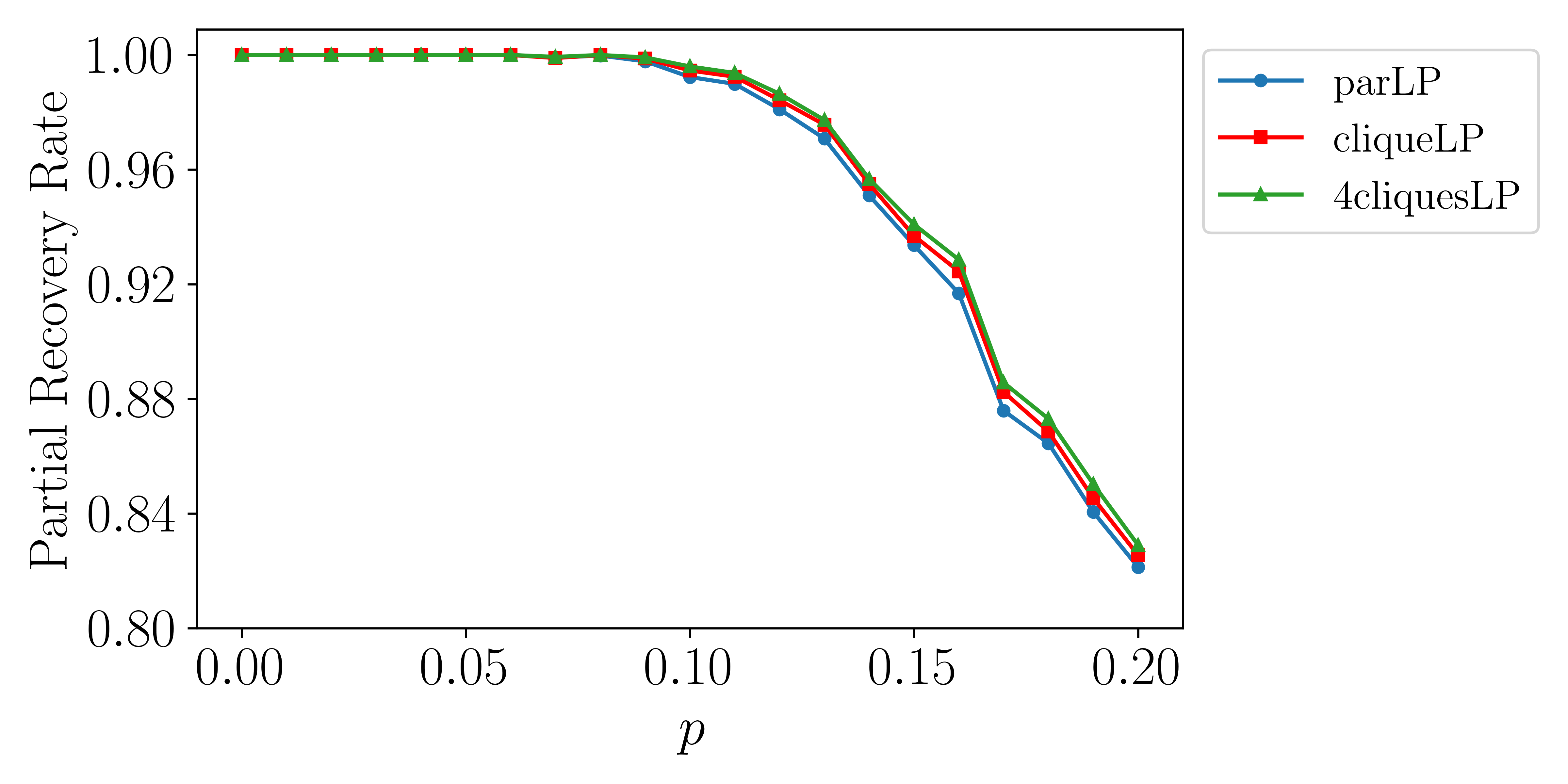}
\caption{$(120,4,3)$ LDPC codes}
\end{subfigure}
\vspace{10pt}
\begin{subfigure}[t]{\textwidth}
\centering
\includegraphics[width=0.48\linewidth]{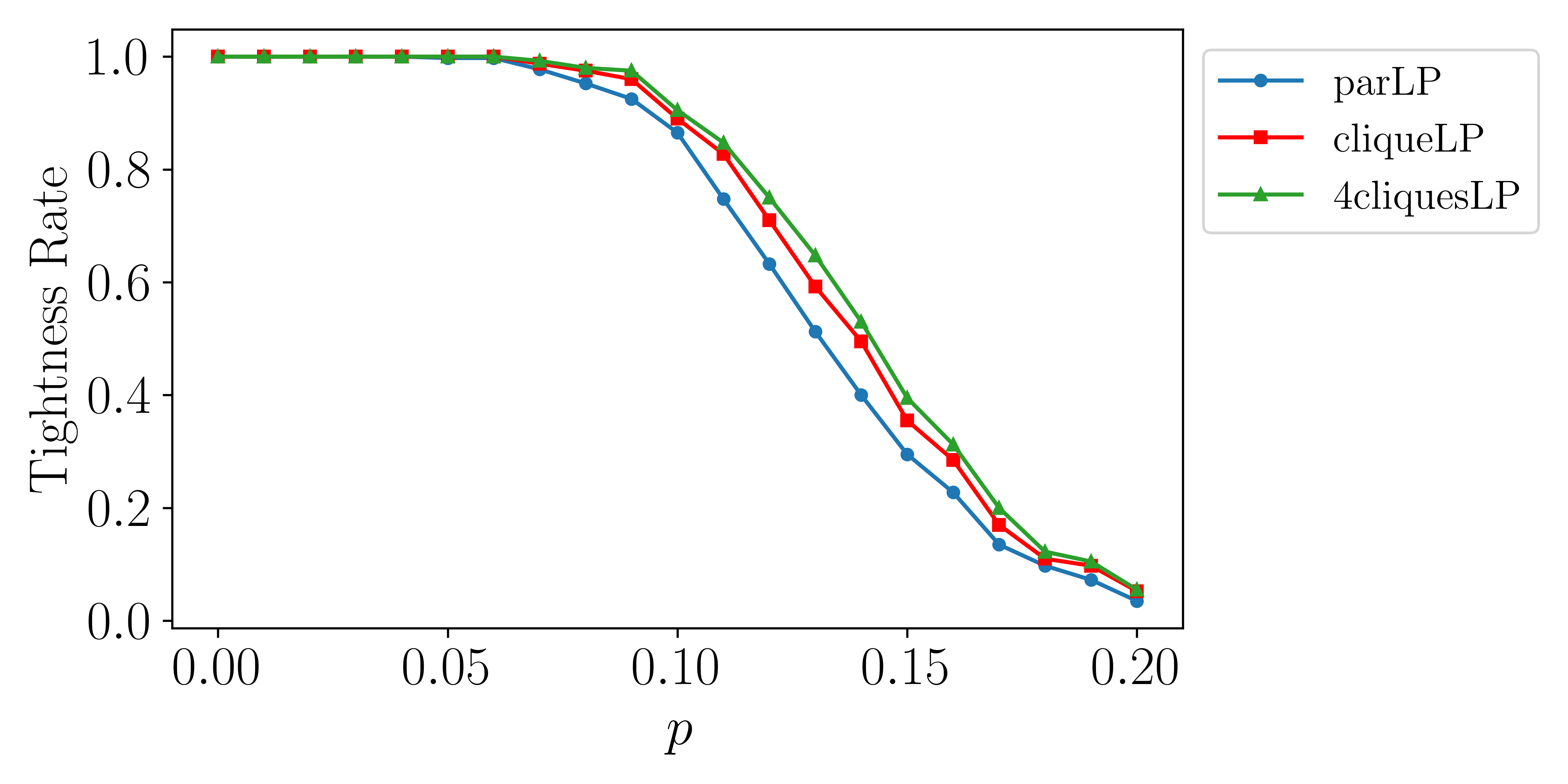}
\includegraphics[width=0.48\linewidth]{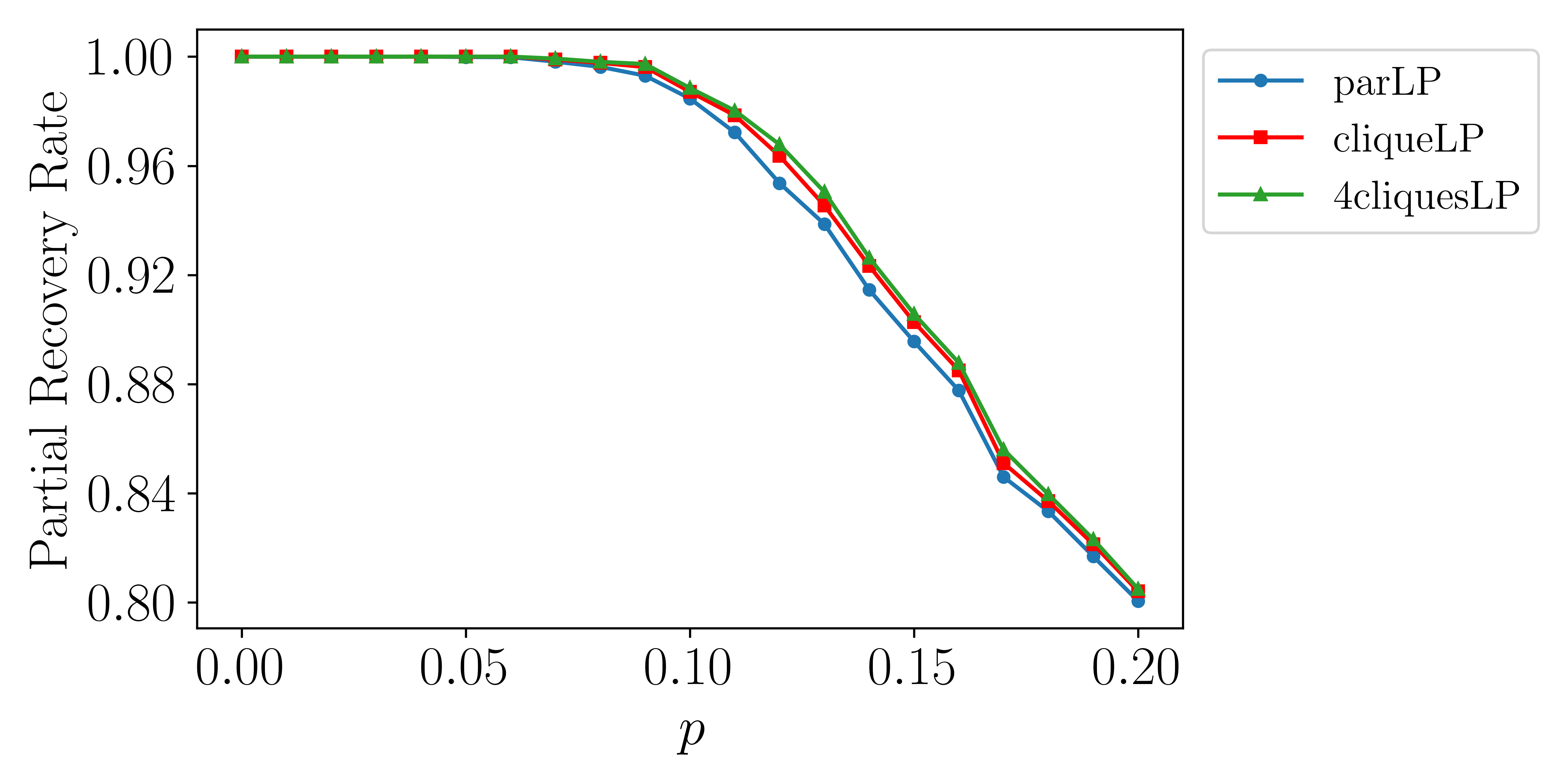}
\caption{$(120,5,4)$ LDPC codes}
\end{subfigure}
\vspace{5pt}
\begin{subfigure}[t]{\textwidth}
\centering
\includegraphics[width=0.48\linewidth]{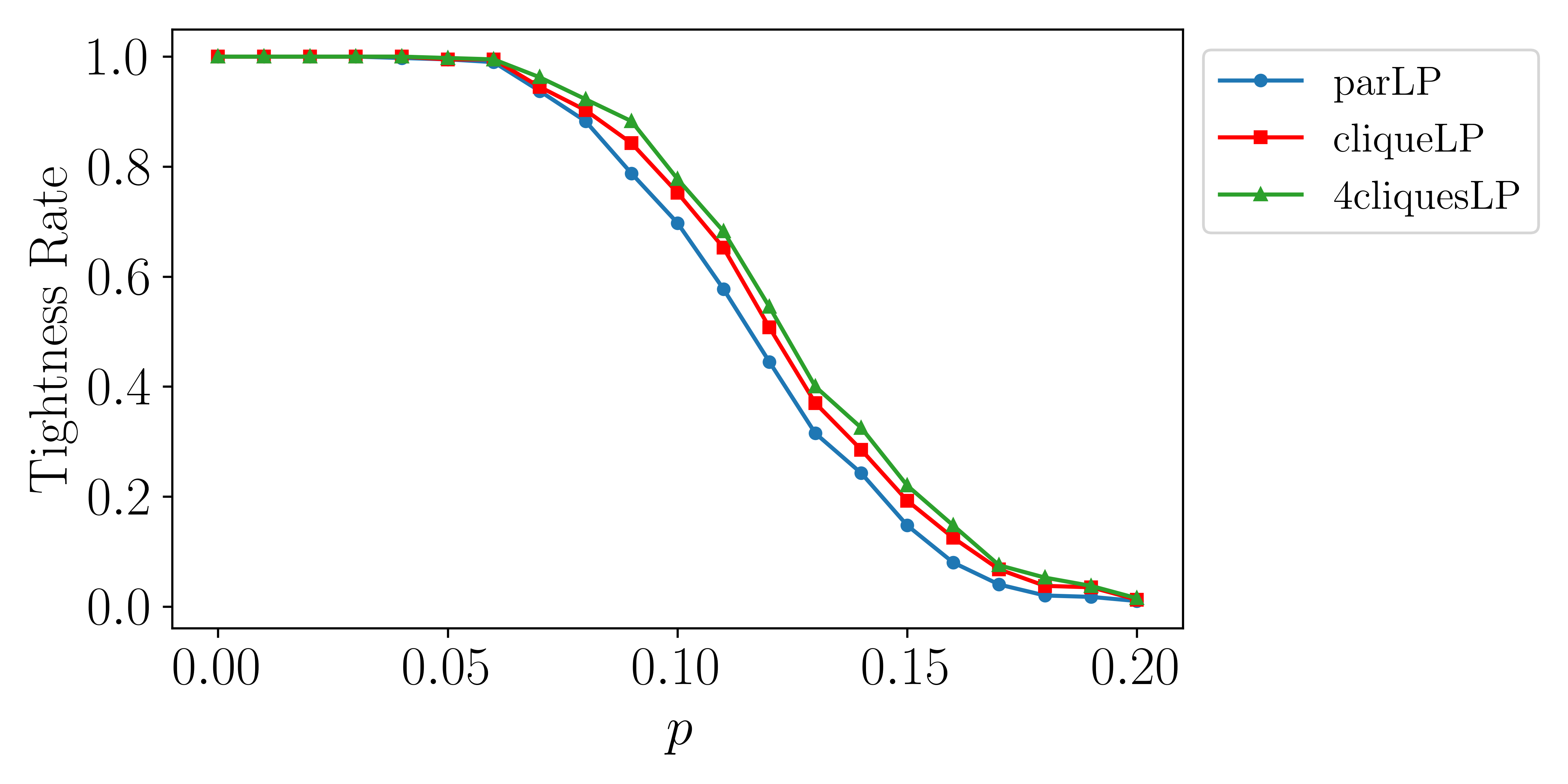}
\includegraphics[width=0.48\linewidth]{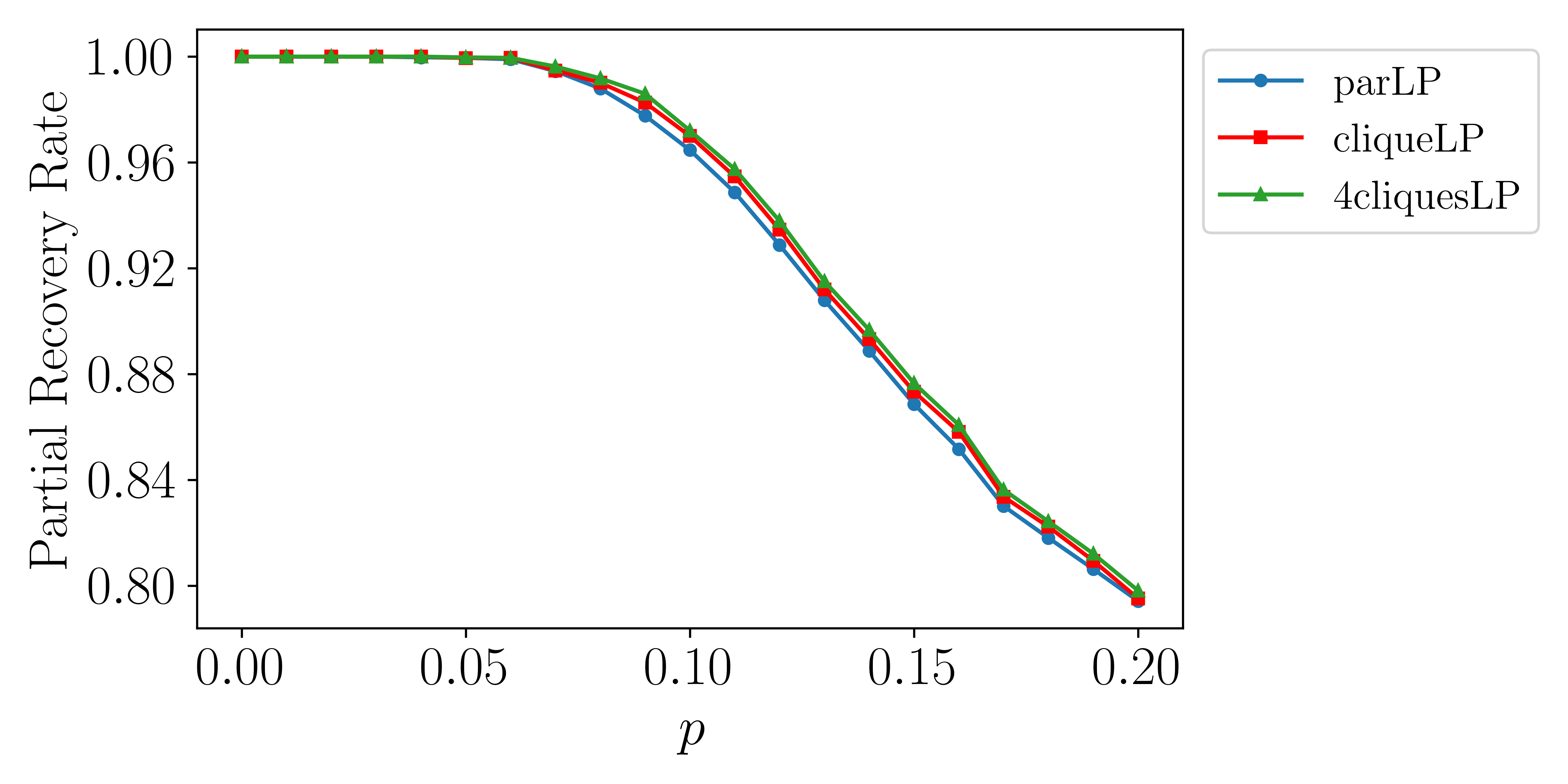}
\caption{$(120,6,5)$ LDPC codes}
\end{subfigure}
\caption{Performance of LP relaxations for decoding
longer LDPC codes.}
\label{fig:longLDPCs}
\end{figure}

\begin{footnotesize}
\bibliographystyle{plain}
\bibliography{biblio,ref}
\end{footnotesize}

\end{document}